\documentclass[12pt]{amsart}
\usepackage[margin=1in]{geometry}
\usepackage{longtable}

\usepackage{aliascnt}
\usepackage{mathrsfs}
\usepackage{amssymb}
\usepackage{enumitem}
\usepackage{xcolor}
\definecolor{darkblue}{rgb}{0.0, 0.0, 0.8}
\usepackage[colorlinks=true, allcolors=darkblue]{hyperref}
\usepackage{tabularx}
\usepackage{mathtools}
\usepackage[capitalise,nameinlink]{cleveref}
\usepackage{autonum}
\usepackage[normalem]{ulem}

\newtheorem{theorem}{Theorem}[section]

\newtheorem{lemma}[theorem]{Lemma}

\newtheorem{assumption}{Assumption}

\crefname{assumption}{Assumption}{Assumptions}

\newtheorem{corollary}[theorem]{Corollary}
\newtheorem{definition}[theorem]{Definition}

\newtheorem{remark}[theorem]{Remark}
\newtheorem{example}[theorem]{Example}
\newtheorem*{notation*}{Notation}
\newtheorem*{theorem*}{Theorem}
\newtheorem*{conjecture*}{Conjecture}

\numberwithin{figure}{section}
\numberwithin{table}{section}

\allowdisplaybreaks

\newcommand{\bN}{\mathbb N}

\newcommand{\bR}{\mathbb R}

\newcommand{\Diff}{\mathrm{Diff}}
\newcommand{\dist}{\operatorname{dist}}

\newcommand{\g}{\ol{g}}

\newcommand{\Imm}{\mathrm{Imm}}

\newcommand{\ol}{\overline}

\newcommand{\Tr}{\mathrm{Tr}}

\newcommand{\vol}{\mathrm{vol}}
\newcommand{\Vol}{\mathrm{Vol}}
\newcommand{\End}{\operatorname{End}}
\newcommand{\Shape}{\mathcal S}
\newcommand{\Group}{\mathcal G}

\newcommand{\pl}{\partial}
\newcommand{\sym}{\operatorname{sym}}
\renewcommand{\H}{\mathcal{H}}
\newcommand{\Christ}[1]{\Gamma_{#1}}
\renewcommand{\Vec}{\sharp}
\newcommand{\Covec}{\flat}
\newcommand{\ChristVec}{\Christ\Vec}
\newcommand{\ChristCovec}{\Christ\Covec}

\def\on{\operatorname}

\newcommand{\inner}[2]{\langle #1,#2 \rangle_{\bR^d}}

\begin{document}

\title[Completeness of invariant  metrics on the space of surfaces]{Completeness of reparametrization-invariant Sobolev metrics on the space of surfaces}
\author{Martin Bauer}
\address{Department of Mathematics, Florida State University}
\email{mbauer2@fsu.edu}

\author{Cy Maor}
\address{Einstein Institute of Mathematics, Hebrew University}
\email{cy.maor@mail.huji.ac.il}

\author{Benedikt Wirth}
\address{Applied Mathematics, University of Münster}
\email{benedikt.wirth@uni-muenster.de}

\begin{abstract}
We study reparametrization-invariant Sobolev-type Riemannian metrics on the space of immersed surfaces and establish conditions ensuring metric and geodesic completeness as well as the existence of minimizing geodesics.
This provides the first extension of completeness results for immersed curves, originating from works of Bruveris, Michor, and Mumford, and validates an earlier conjecture of Mumford on completeness properties of general spaces of immersions in this important case.

The result is obtained by recasting earlier approaches to completeness on manifolds of mappings as a general completeness criterion for infinite-dimensional Riemannian manifolds that are open subsets of a complete Riemannian manifold, and 
by combining it with geometric estimates based on the Michael--Simon--Sobolev inequality to establish the completeness for specific Sobolev metrics on immersed surfaces.
We expect that this approach will be useful for obtaining completeness results for other manifolds of mappings.
\end{abstract}

\maketitle
\setcounter{tocdepth}{1}
\tableofcontents

\section{Introduction}
Starting with the influential work of Younes~\cite{younes1998computable}, reparametrization-invariant Riemannian metrics on spaces of immersions have become a topic of significant research interest, see, e.g.,~\cite{younes2010shapes,srivastava2016functional,schmeding2022introduction,michor2007overview,bauer2014overview} and the references therein. The study of these geometries can be motivated from several perspectives: First, they  can be viewed as the natural generalization of right-invariant metrics on groups of diffeomorphisms, which build the foundation of Arnold's geometric viewpoint for many prominent equations in hydrodynamics~\cite{arnold2009topological,constantin2003geodesic,marsden2013introduction}. As such they provide a canonical class of examples of infinite-dimensional Riemannian manifolds that might allow us to gain a deeper understanding of the phenomena and difficulties arising in the context of infinite-dimensional geometry, including completeness properties and the failure of the theorem of Hopf--Rinow~\cite{atkin1975hopf,ekeland1978hopf} or vanishing geodesic distance and locally unbounded curvature~\cite{michor2005vanishing,eliashberg1993bi}.
Equally significant is their role in geometric data science and mathematical shape analysis, where one seeks mathematically rigorous methods for comparing geometric objects --- such as immersed curves, surfaces, or, more generally, submanifolds --- in a mathematically precise manner. Reparametrization-invariant metrics on spaces of immersions descend to Riemannian metrics on the shape space of (unparametrized) geometric objects and consequently the induced geodesic distance provides a natural basis for defining a statistical framework on this infinite-dimensional space~\cite{younes2010shapes,srivastava2016functional,bauer2014overview}.

In this article, we are interested in completeness properties of reparametrization-invariant Riemannian metrics on the space of immersed surfaces.
In addition to their theoretical importance, completeness and existence of minimizing geodesics are of particular relevance in the aforementioned applications in shape analysis, as they guarantee the well-posedness of the resulting shape matching algorithms, which serve as the basis for any subsequent geometric statistics analysis. To the best of our knowledge, completeness properties for Riemannian metrics on spaces of immersions were first discussed by David Mumford around 2013: In his presentation at the annual meeting of the NSF-funded Focused Research Group  ``The geometry, mechanics and statistics of the infinite-dimensional manifold of shapes'' 
he conjectured geodesic completeness of the space of immersions equipped with  Sobolev metrics of sufficiently high order; this would imply that these geometries exhibit a completeness behavior analogous to that of right-invariant Sobolev metrics on diffeomorphism groups, for which completeness of strong enough metrics is known --- a line of research that began with the seminal work of Ebin and Marsden~\cite{ebin1970groups} and has seen important developments in recent years~\cite{misiolek2010fredholm,bruveris2017completeness,escher2014geodesic,bauer2015local,bauer2025regularity}.

In 2014, Bruveris, Michor, and Mumford~\cite{bruveris2014geodesic} confirmed this conjecture for the space of plane curves equipped with a Sobolev metric of order two or higher. Since then, their result has been extended in several directions, encompassing curves in~$\mathbb{R}^d$~\cite{bruveris2015completeness} as well as addressing questions of metric completeness and the existence of minimizing geodesics~\cite{bruveris2015completeness,nardi2016geodesics}.
Further developments have covered weighted metrics \cite{bruveris2017completenessB}, Sobolev metrics of non-integer order~\cite{bauer2024completeness} and curves with values in general manifolds~\cite{bauer2023sobolev}.
Notably, while the first completeness paper~\cite{bruveris2014geodesic} reiterates the conjecture that higher order metrics on surfaces are expected to be complete, all of these advances pertain exclusively to immersed curves. 
Prior to the present work, the corresponding completeness problem for the arguably most significant case --- the space of immersed surfaces --- remained entirely unresolved. The main result of this article addresses this open question
and gives the first construction of complete, invariant metrics on the space of surfaces. 
\subsection{Contributions of the article}
In the following we introduce the main concepts and results of the article. For complete definitions we refer the reader to \cref{sec:spaces}.

In all of this article, let $M$ be a two-dimensional, oriented, closed manifold.  We consider the space $\Imm^l(M,\bR^d)$ of all immersions of  $M$ into $\bR^d$ with $d\geq 3$, of integer Sobolev regularity $H^l$.
Here we assume $l\geq 3$ so that the Sobolev space $H^l(M,\bR^d)$ embeds into $C^1(M,\bR^d)$, hence the immersion condition is well-defined. 
We equip this infinite-dimensional manifold with a curvature-weighted Sobolev metric of integer order $3\leq k\leq l$ of the form
\begin{align}\label{eq:curvatureweightedmetricIntro}
G^k_f(h_1,h_2)
= \int_{M} h_1\cdot h_2+|{\H}|^4&\left(g(\nabla h_1,\nabla h_2)+g(\nabla^2 h_1,\nabla^2 h_2)\right)
+ g(\nabla^k h_1,\nabla^k h_2)\,\vol.
\end{align}
Here $f\in \Imm^l(M,\bR^d)$, $h_1, h_2\in T_f\Imm^l(M,\bR^d)\cong H^l(M,\bR^d)$ and  $g, \H, \nabla$ and $\vol$ denote the surface metric, mean curvature, covariant derivative and volume density induced by $f$. 
We also introduce the shape space of unparametrized surfaces $\Shape^l(M,\bR^d):=\Imm^l(M,\bR^d)/\Diff^l(M)$, where
$\Diff^l(M)$ denotes the group of orientation-preserving diffeomorphisms of Sobolev class $H^l$. Since the Riemannian metric $G^k$ is invariant under the action of $\Diff^l(M)$ it induces a corresponding geometry on the shape space.
The main result of this article is as follows:
\begin{theorem}[Completeness Properties of the Space of Immersed Surfaces]\label{theorem:main}
Let $M$ be a two-dimensional, closed manifold, let $d,l\geq 3$ and let $\Imm^l(M,\bR^d)$  be the space of all immersions of $M$ into $\bR^d$ of regularity $H^l$, equipped with the Riemannian metric $G^k$ from~\eqref{eq:curvatureweightedmetricIntro} of order $3\leq k\leq l$. We have:
\begin{enumerate}[label=(\alph*)]
\item {\bf Metric  Completeness of $\Imm^k(M,\bR^d)$:} The space $(\Imm^k(M,\bR^d), G^k)$ is metrically complete.\label{item:metriccomplete}
\item {\bf Geodesic Completeness of $\Imm^l(M,\bR^d)$ for $l\geq k$:} The space $(\Imm^l(M,\bR^d), G^k)$ is geodesically complete. The result continues to hold for $l=\infty$, i.e., the space of smooth immersions $\Imm(M,\bR^d)$ equipped with the metric $G^k$ is geodesically complete.\label{item:geodesiccomplete}
\item {\bf Existence of Minimizing Geodesics on $\Imm^k(M,\bR^d)$:} For any two immersed surfaces $f_0,f_1$ in the same connected component of $\Imm^k(M,\bR^d)$ there exists a minimizing $G^k$-geodesic in $\Imm^k(M,\bR^d)$ connecting them.
\label{item:geodesicconvex}
\item {\bf Metric  Completeness of $\Shape^k(M, \mathbb{R}^d)$:} The space $\Shape^k(M, \mathbb{R}^d)$ is metrically complete with respect to the metric induced by $G^k$.
 \item {\bf Existence of Optimal Reparametrizations:} For any $[f_1], [f_2] \in \Shape^k(M, \mathbb{R}^d)$ in the same connected component
there exists an optimal reparametrization $\bar \varphi\in \Diff^k(M)$
attaining the infimum in the definition of the quotient distance,
\begin{equation}
\dist([f_1], [f_2])
  = \inf_{\varphi \in \Diff^k(M)} 
    \dist_{G^k}(f_1, f_2 \circ \varphi)=
    \dist_{G^k}(f_1, f_2 \circ \bar \varphi).
\end{equation}
\item {\bf Shape Space is a Geodesic Length Space:} For any two elements $[f_1], [f_2] \in \Shape^k(M, \mathbb{R}^d)$ in the same connected component
there exists a minimizing connecting geodesic in $\Shape^k(M, \mathbb{R}^d)$, in the sense of metric geometry, for the metric induced by $G^k$.
\end{enumerate}
\end{theorem}
Note that, using density arguments, this result also implies that the metric completion of
the space $(\Imm^l(M,\bR^d), G^k)$ for $3\leq k<l$ is precisely given by the space $(\Imm^k(M,\bR^d), G^k)$ (and similarly for the shape space $(\Shape^l(M,\bR^d), G^k)$).

The proof of this theorem follows from more general Theorems --- \cref{thm:metric_completeness_abstract}, \cref{thm:existence_abstract}, \cref{thm:metric_complete_shape} and \cref{thm:metric_convex_shape} --- which are formulated in terms of an abstract Riemannian metric on the space of immersions satisfying certain $L^p$-bounds.
The mean curvature weights in $G^k$ are used to ensure that these bounds hold, using
various Michael--Simon--Sobolev type estimates~\cite{michael1973sobolev}, which is shown in~\Cref{lemma:assumptions}.
In particular, the result holds for any metric of order $k$ that controls $G^k$.
As discussed in~\cref{rem:otherchoices} the metric $G^k$ is not necessarily the optimal choice in terms of the chosen powers of the mean curvature weight; it is merely the most compact to write down, which is why we decided to present the results for this particular choice of weights.
Furthermore, we emphasize that, as shown by Atkin and Ekeland~\cite{atkin1975hopf,ekeland1978hopf}, the theorem of Hopf--Rinow is invalid in this infinite-dimensional setting, and thus the existence of minimizing geodesics is not an automatic consequence of the metric completeness, but has to be proven separately.

To prove these results we develop in \cref{thm:abstract_completeness} a general criterion for inheritance of completeness and existence of minimizing geodesics on abstract Riemannian manifolds that are open subsets of a complete manifold, where the completeness of the larger space is with respect to a different metric.
In our case, the larger space is the Hilbert space $H^l(M,\bR^d)$.
This criterion is a generalization of the strategy used for the completeness results on the space of curves~\cite{bruveris2015completeness,bauer2023sobolev,bauer2024completeness}.

We believe that this result is of independent interest and will be useful in other important cases.
For example, the space of all $H^k$-probability densities or the space of $H^k$-Riemannian metrics can naturally be regarded as open subsets of a Hilbert space. 
To illustrate its practicality, we further examine a simple, synthetic example involving a particular submanifold within the space of $\ell^2$-sequences, cf.~\cref{ex:l2}.

In our situation of $\Imm^l(M,\bR^d)$, the completeness criterion (essentially) reduces the completeness of the space of immersions to controlling various geometric quantities --- such as the metric, mean curvature, and their derivatives --- on $G^k$-metric balls. 
These estimates, which are the main technical contribution of the paper, are presented in ~\cref{sec:bounds}. Once these estimates are established the metric and geodesic completeness follow readily from the abstract result.
This is done in \cref{sec:completeness}.
The existence of minimizing geodesics requires one to check an additional technical assumption, which we show in~\cref{sec:existence} to be satisfied.

\subsection{Open Questions and Future Work}
The results of this article lead us to the following conjecture regarding completeness properties of reparametrization-invariant Sobolev metrics on spaces of immersions of submanifolds of general dimension, which is a more precise version of the conjecture of Mumford mentioned above:
\begin{conjecture*}[Completeness Properties on the Space of Immersed Submanifolds]
Let $M$ be an $m$-dimensional, closed manifold and let $(N,\bar g)$ be a $d$-dimensional, Riemannian manifold of bounded geometry with $d>m$. Let $\Imm^l(M,N)$  denote the Sobolev space of all immersed submanifolds of type $M$ in $N$,  where $l>\frac{m}2+1$.

For any (possibly non-integer) $k>\frac{m}{2}+1$ there exists a reparametrization-invariant, mean-curvature weighted Sobolev metric $G^k$ of order $k$ such that the following holds:
\begin{enumerate}[label=(\alph*)]
\item {\bf Metric  Completeness of $\Imm^k(M,N)$:} The space $(\Imm^k(M,N), G^k)$ is metrically complete.
\item {\bf Geodesic Completeness of $\Imm^l(M,N)$ for $l\geq k$:} The space $(\Imm^l(M,N), G^k)$ is geodesically complete. The result continues to hold for $l=\infty$, i.e., the space of smooth immersions $\Imm(M,N)$ equipped with the metric $G^k$ is geodesically complete.
\item {\bf Existence of Minimizing Geodesics on $\Imm^k(M,N)$:} For any two immersed surfaces $f_0,f_1$ in the same connected component of $\Imm^k(M,N)$ there exists a minimizing $G^k$-geodesic in $\Imm^k(M,N)$ connecting them.
\item {\bf Metric  Completeness of $\Shape^k(M, N)$:} The space $\Shape^k(M, N)$ is metrically complete with respect to the metric induced by $G^k$.
 \item {\bf Existence of Optimal Reparametrizations:} For any $[f_1], [f_2] \in \Shape^k(M, N)=\Imm^k(M,N)/\Diff^k(M)$ in the same connected component,
there exists an optimal reparametrization $\bar \varphi\in \Diff^k(M)$
attaining the infimum in the definition of the quotient distance.
\item {\bf Shape Space is a Geodesic Length Space:} For any two elements $[f_1], [f_2] \in \Shape^k(M, N)$ in the same connected component,
there exists a minimizing connecting geodesic in $\Shape^k(M, N)$, in the sense of metric geometry, for the metric induced by $G^k$.
\end{enumerate}

\end{conjecture*}
We believe that the results of this paper establish a clear path toward proving this conjecture, at least for integer $k$ and  Euclidean ambient space; in particular, the general criterion \cref{thm:abstract_completeness} and most of the estimates of \cref{sec:bounds} work in this generality.
The case of a non-Euclidean ambient space is more challenging, cf.\ the completeness for manifold-valued curves~\cite{bauer2023sobolev} as compared to the same results for curves with values in Euclidean space~\cite{bruveris2015completeness}.
Generalizing to non-integer $k$ seems to be even more delicate, as it seems that the theory of fractional (non-integer) Michael--Simon--Sobolev inequalities is not developed for arbitrary manifolds, cf.~\cite{cabre2022fractional}.

An important generalization of this conjecture, which seems quite challenging at the moment, is characterizing the weights that ensure completeness;
in particular, we do not know if non-weighted Sobolev metrics on $\Imm^k(M,\bR^d)$ are complete.
That is, we do not know whether the intermediate terms in $G^k$ as defined above are essential. Finally, for $k<\frac{m}{2}+1$, geodesic incompleteness is expected and is known in some cases (e.g., non-weighted metrics), but the question is open whether it is true for arbitrary weights.

\noindent {\bf Going beyond the Riemannian Case.}
A natural extension of the present results is to consider Finsler metrics. 
Notably, since the proofs of metric completeness and existence of minimizing geodesics involve estimates that are based purely on norm properties, without any reliance on the inner product structure, we expect that at least these extend to Finsler metrics in a rather straightforward way.
In particular, this would allow one to construct a Finsler-type Sobolev metric of order two that is metrically complete on the space of $W^{2,p}$-immersions with $p>2$.
It is less clear whether the geodesic completeness results would similarly translate to this setting.

\noindent {\bf Completeness on the Space of Embeddings.}
Another natural extension of this work is to the space of embeddings. 
As embeddedness is a non-local condition and thus cannot be directly detected with the metrics studied in this article, these metrics are incomplete when restricted to embeddings.
For the case of embedded curves, this was recently overcome by Reiter, Schumacher and collaborators~\cite{reiter2021sobolev,dohrer2025complete} by adding certain terms to metric of a non-local type.
We envision that the results of the present article could open up the path to a similar result in the context of the space of embedded surfaces.

\noindent {\bf Convergent discretizations.}
Numerical applications require geodesics in $\Imm^k(M,\bR^d)$ to be discretized in spacetime,
including the associated geodesic equations or the variational problem of finding minimizing geodesics.
While it is relatively straightforward to show consistency of a given discretization,
the convergence of discretely computed geodesics to continuous ones for increasing resolution is more delicate.
It typically requires mimicking compactness and completeness arguments from the continuous setting.
In the case of curves this was achieved in \cite{beutler2025discretegeodesiccalculusspace}
by exploiting a discrete path energy that controlled the continuous path energy.
A natural task is to devise similar approaches in our setting of surfaces.

 \subsection{Structure of the Article}
In \cref{sec:completenessAbstract} we present our criterion for completeness properties of Riemannian  manifolds which are an open subset of another complete Riemannian manifold. Next, we introduce the space of surfaces and Riemannian metrics thereon in \cref{sec:spaces}, followed by a review of certain Michael--Simon--Sobolev inequalities in \cref{sec:SobolevEmbeddings}. The main technical estimates on bounding intrinsic geometric quantities on metric balls are presented in \cref{sec:bounds} and are subsequently used in \cref{sec:completeness} to obtain the desired completeness results.
Finally, in \Cref{sec:curvature_weighted} we show that the curvature weighted $H^k$-metric satisfies the assumptions of the abstract theorems.
A table of the most relevant notation, as used throughout the article, is presented in \cref{app:table_notation}.

\subsection*{Funding}
This work was supported by the Deutsche Forschungsgemeinschaft (DFG, German Research Foundation) via project 431460824 -- Collaborative Research Center 1450 as well as via Germany’s Excellence Strategy project 390685587 -- Mathematics Münster: Dynamics--Geometry--Structure.
Part of this work was written when CM was visiting the University of Toronto; CM is grateful for their hospitality. 
MB was partially supported by NSF grants 2426549 and 2526630.
CM was partially supported by ISF grant 2304/24. 
MB and CM were partially supported by BSF grant 2022076. Parts of this work was 
written when MB was attending the program
Infinite-dimensional Geometry: Theory and Applications at the Erwin Schr\"odinger International
Institute for Mathematics and Physics (ESI). The authors sincerely
thank the organizers of the program and the ESI for fostering a stimulating environment that encouraged fruitful discussions and enabled this work.
\subsection*{Acknowledgements}
The first author has worked on several cases of the completeness conjecture for spaces of immersions for the better part of the last ten years. During this period, there were numerous (unsuccessful) attempts to establish the result in the context of the present paper. The author wishes to express his gratitude to all those who patiently discussed and commented on these various efforts. This includes, in particular, Martins Bruveris, Philipp Harms, Peter Michor, Martin Rumpf and all regular participants of the annual workshop “Math en plein air”. Their collective feedback and insights have been instrumental in developing the strategy that ultimately led to the successful completion of this work. In addition we sincerely thank Fabian Rupp, who pointed us towards the work of Kuwert and Sch\"atzle~\cite{kuwert2002gradient}, which allowed us to overcome one of the major road blocks.

\section{Completeness for Strong Riemannian Metrics}\label{sec:completenessAbstract}
In this section we present an abstract method for  obtaining completeness of strong Riemannian  manifolds which are an open subset of a complete Riemannian manifold. We will later use this result, which is based on a method developed in~\cite{bauer2015metrics,bruveris2015completeness,bauer2023sobolev} for the case of immersed curves, to obtain the desired completeness result for the space of surfaces. In recent work~\cite{bauer2025regularity}, an analogous result for infinite-dimensional half-Lie groups endowed with right-invariant metrics was obtained. 
Half-Lie groups are smooth manifolds and topological groups such that right translations are smooth, but left translations are merely required to be continuous~\cite{marquis2018half}; this typically applies to diffeomorphism groups of finite regularity.
Our setting (immersions rather than diffeomorphisms) is considerably more involved, since we cannot rely on the simplifications provided by a group structure; this difficulty is addressed by imposing the existence of a complete ambient manifold, bounding the given Riemannian metric as an additional assumption.

\begin{theorem}[Inheritance of Completeness Properties]\label{thm:abstract_completeness}
Let $(\ol{\mathcal M},\ol G)$ be a metrically complete smooth strong Riemannian manifold.
Let $\mathcal M\subset \ol{\mathcal M}$ be open, and let $G$ be a smooth strong metric on $\mathcal M$. Assume that on any $G$-metric ball $B$ in $(\mathcal M,G)$ the following hold:
\begin{enumerate}[label=(\alph*)]
\item \label{assumption1} There exists $C>0$ such that for all $x\in B$ and $h\in T_x\mathcal M$ we have \[\ol G_x(h,h) \leq C G_x(h,h).\]
\item \label{assumption2} The closure of $B$ in $\ol{\mathcal M}$ (with respect to $\ol{G}$) lies in $\mathcal M$.
\end{enumerate}
Then $(\mathcal M,G)$ is metrically and geodesically complete.
Let furthermore $(\ol{\mathcal M},\ol G)$ be a Hilbert space, and assume the following condition:
\begin{enumerate}[label=(\alph*), start=3]
\item \label{assumption3}
The path energy associated with $G$,
\begin{equation}
E:H^1((0,1),(\ol{\mathcal{M}},\ol{G}))\to[0,\infty],\qquad E(f)=\int_0^1G_{f_t}(\dot f_t,\dot f_t)\,\mathrm d t,
\end{equation}
is weakly sequentially lower semicontinuous along sequences of paths with fixed end points in $\mathcal M$ (where we set $G_{f}(h,h)=\infty$ for $f\notin\mathcal M$).
\end{enumerate}
Then there exists a $G$-minimizing geodesic between any two points $f_0,f_1$ at finite $G$-distance, which is equivalent to $f_0,f_1$ lying in the same connected component of $\mathcal{M}\subset\ol{\mathcal M}$.
\end{theorem}

\begin{remark}[Metric Completeness Implies Geodesic Completeness]
Note that the Theorem of Hopf--Rinow is not valid in infinite dimensions~\cite{atkin1975hopf,ekeland1978hopf}. In particular, metric and geodesic completeness are not equivalent, and neither implies existence of minimizing geodesics.
Nevertheless, for smooth strong metrics, metric completeness still implies geodesic completeness \cite[VIII, Prop.\,6.5]{lang2012fundamentals}, and thus in \cref{thm:abstract_completeness} it is sufficient to prove metric completeness and existence of minimizing geodesics.
\end{remark}

\begin{proof}
\textbf{Metric completeness.} Let $B=B_r(x_0)$ be a metric ball of radius $r>0$ in $(\mathcal M,G)$, and let $x_1,x_2\in B$ be two distinct points.
There exists a curve $\gamma$ connecting $x_1$ and $x_2$ whose $G$-length $\ell_G$ is less than $2\dist_G(x_1,x_2)<4r$ (note that the geodesic distance is positive between different points since $G$ is a strong metric).
In particular, this curve must lie in $B_{3r}(x_0)$, since each point in $\gamma$ is of distance at most $\ell_G/2<2r$ from either $x_1$ or $x_2$.
Now, assumption~\ref{assumption1} implies that the $\ol{G}$-length of $\gamma$, denoted by $\ell_{\ol{G}}$, satisfies $\ell_{\ol{G}}\le \sqrt{C} \ell_G$, where $C$ is the constant in assumption~\ref{assumption1} of the ball $B_{3r}(x_0)$.
Thus,
\[
\dist_{\ol{G}}(x_1,x_2) \le \ell_{\ol{G}} \le \sqrt{C} \ell_G < 2\sqrt{C}\dist_G(x_1,x_2),
\]
so we obtain that the $\ol{G}$-distance is controlled by the $G$-distance on $G$-metric balls.

Now, let $(x_j)_{j\in\mathbb{N}}$ be a Cauchy sequence in $(\mathcal M,G)$.
Since it is a Cauchy sequence, it lies in some $G$-metric ball $B$.
By what we showed above, it follows that it is also a Cauchy sequence in $(\ol{\mathcal M},\ol{G})$.
Since this space is metrically complete, we have that $x_j\to x_0$ in $(\ol{\mathcal M},\ol{G})$ for some $x_0\in \ol{\mathcal M}$;
in particular we have that $x_0$ is in the $\ol{G}$-closure of $B$.
Therefore, by assumption~\ref{assumption2} we have that $x_0\in \mathcal M$, i.e., $x_j \to x_0$ in $(\mathcal M,\ol{G})$.
Since both $G$ and $\ol{G}$ are strong metrics on $\mathcal M$, they induce the same topology, and thus $x_j\to x_0$ in $(\mathcal M,G)$, which completes the proof.

\noindent\textbf{Connected components.}
We now show that the connected components of $\mathcal M$, as a subset of $\ol{\mathcal M}$, coincide with the connected components induced by $\dist_G$ --- that is, $f_0,f_1\in\mathcal M$ are in the same connected component if and only if $\dist_G(f_0,f_1)<\infty$.

Since $G$ is a strong Riemannian metric it induces the same topology on each tangent space as the original manifold topology \cite[VII Prop.\,6.1]{lang2012fundamentals}. 
Thus, $G$ induces the same topology on $\mathcal M$ as the original topology (the induced topology as a subset of $\ol {\mathcal M}$).
In particular, the connected components are the same, and are open.

Points in connected open subsets of smooth manifolds can be connected by smooth (or piecewise-smooth) paths, and such a path $f:[0,1]\to\mathcal{M}$ has finite path energy $E(f)$
by the smoothness of the metric.
Therefore any $f_0,f_1$ from the same connected component can be joined by a path of finite path energy. 
Thus, each $f_0, f_1$ in the same connected component have finite distance $\dist_{G}(f_0,f_1)=\inf\{\sqrt{E(f)}:\,f\text{ connects }f_0,f_1\}$.
Conversely, $\dist_{G}(f_0,f_1)<\infty$ implies that $f_0,f_1$ lie in the same path-connected component and thus in the same connected component.

\noindent\textbf{Minimizing geodesics.}
We will prove the existence of minimizing geodesics by the direct method of the calculus of variations.
Consider a minimizing sequence $f^n$, $n\in\bN$, of paths connecting $f_0$ to $f_1$ with finite, monotonously decreasing path energy $E(f^n)\to\dist_{G}(f_0,f_1)^2$,
then all these paths lie within a single, sufficiently large $G$-metric ball $B$.
By assumption~\ref{assumption1}, we have
$$\|\dot f^n\|_{L^2((0,1),(\ol{\mathcal{M}},\ol{G}))}^2=\int_0^1\ol G^k(\dot f^n,\dot f^n)\,\mathrm d t\lesssim E(f^n)\leq E(f^0)<\infty$$ for all $n$.
Due to $f^n_s=f_0+\int_0^s\dot f^n_t\,\mathrm d t$ we additionally have
\begin{align}
\|f^n\|_{L^2(\ol{\mathcal{M}},\ol{G}))}^2
&\leq2\|f_0\|_{\ol{G}}^2+2\!\int_0^1\!\left\|\int_0^s\!\!\dot f^n_t\,\mathrm d t\right\|_{\ol{G}}^2\mathrm d s
\leq2\|f_0\|_{\ol{G}}^2+2\!\int_0^1\!\!\left(\int_0^s\!\!\|\dot f^n_t\|_{\ol{G}}\,\mathrm d t\right)^{\!\!2}\mathrm d s\\
&\leq2\|f_0\|_{\ol{G}}^2+2\!\int_0^1\left(\int_0^1\|\dot f^n_t\|_{\ol{G}}\,\mathrm d t\right)^2\mathrm d s
\leq2\|f_0\|_{\ol{G}}^2+2\|\dot f^n\|_{L^2((0,1),(\ol{\mathcal{M}},\ol{G}))}^2,
\end{align}
and $f^n$ is uniformly bounded in $H^1((0,1),(\ol{\mathcal{M}},\ol{G}))$.
Since $H^1((0,1),(\ol{\mathcal{M}},\ol{G}))$ is a Hilbert space, there exists a subsequence, for simplicity still indexed by $n$,
with $f^n\rightharpoonup f$ weakly in $H^1((0,1),(\ol{\mathcal{M}},\ol{G}))$ for some $f\in H^1((0,1),(\ol{\mathcal{M}},\ol{G}))$.
Moreover, by $\|f_t-f_s\|_{\ol{G}}=\|\int_t^s\dot f_r\,\mathrm d r\|_{\ol{G}}\leq\int_t^s\|\dot f_r\|_{\ol{G}}\,\mathrm d r\leq\sqrt{t-s}\|\dot f\|_{L^2((0,1),(\ol{\mathcal{M}},\ol{G}))}$ for any $0\leq t\leq s\leq1$,
the space $H^1((0,1),(\ol{\mathcal{M}},\ol{G}))$ continuously embeds into the H\"older space $C^{0,1/2}([0,1],(\ol{\mathcal{M}},\ol{G}))$
so that $f^n\rightharpoonup f$ weakly also in that space.
As a consequence we have pointwise weak convergence $f_t^n\rightharpoonup f_t$ in $(\ol{\mathcal M},\ol G)$ for all $t\in[0,1]$.
In particular, the limit $f$ is still a continuous path in $\ol{\mathcal M}$ connecting $f_0$ with $f_1$.
The lower semicontinuity condition \ref{assumption3} now directly implies
\begin{equation}
\dist_G(f_0,f_1)^2
=\liminf_{n\to\infty} E(f^n)
\geq E(f).
\end{equation}
In particular, it follows that $f_t\in \mathcal{M}$ for almost every $t$, and since the path $t\mapsto f_t$ is continuous in $\ol{\mathcal{M}}$, it follows from assumption~\ref{assumption2} that $f_t\in \mathcal{M}$ for every $t\in [0,1]$, hence $f$ is a valid path in $\mathcal{M}$ from $f_0$ to $f_1$, and thus, by the inequality above, a minimizing geodesic.
\end{proof}

\begin{remark}[Smoothness Requirements]
    Note that the smoothness of the Riemannian metric is not a necessary condition for the result. Any regularity which guarantees that metric completeness implies geodesic completeness and that the topology induced by the metric coincides with the manifold topology would be sufficient.
\end{remark}

\begin{example}[$\ell_2$-Sequences Bounded Away from $-1$]\label{ex:l2}
We now show a simple example illustrating the use of \cref{thm:abstract_completeness}.
Let $\ol{\mathcal M} = \ell_2$ be the Hilbert space of real square-summable sequences, and let $\ol{G}$ be its standard metric.
Let 
\[\textstyle
\mathcal M = \{x\in \ell_2 ~:~ \inf_{i\in \mathbb{N}} x_i > -1\},
\]
which is obviously an open subset of $\ol{\mathcal M}$, hence its tangent space at every point can be identified with $\ell_2$.
Endow it with the metric
\[
G_x(h,h) = \sum_{i\in \mathbb{N}}\left(1 + \frac{1}{(x_i+1)^2}\right) h_i^2.
\]
Note that this is a well-defined metric by the definition of $\mathcal M$ and that obviously $\ol{G}\le G$ in $\mathcal M$, so assumption~\ref{assumption1} trivially holds.

To prove assumption~\ref{assumption2}, we consider the map $F:\mathcal M\to \ell_\infty$,
$F(x) = (\log(x_i+1))_{i\in \mathbb{N}}$.
We have that
\[
\|T_xF[h]\|_\infty = \left\| \left(\frac{h_i}{x_i+1}\right)_{i\in \mathbb{N}}\right\|_\infty \le \left(\sum_{i\in \mathbb{N}}\frac{h_i^2}{(x_i+1)^2}\right)^{1/2}\le \sqrt{G_x(h,h)}.
\]
Hence $F$ is $1$-Lipschitz and in particular bounded on $G$-metric balls.
It follows that, for every $G$-metric ball $B\subset \mathcal M$, we have
\[
\inf_{x\in B}\inf_{i\in \mathbb{N}} (x_i+1) > 0,
\]
and any $\ol{G}$-limit point of a sequence in $B$ is also in $\mathcal M$.
Thus assumption~\ref{assumption2} is satisfied, and $(\mathcal M,G)$ is metrically  and geodesically complete.

To prove assumption~\ref{assumption3}, consider a sequence $x^n\rightharpoonup x$ converging weakly in $H^1((0,1),\ell^2)$ with fixed $x^n(0),x^n(1)\in\mathcal M$.
We can restrict ourselves to the connected component of positive sequences.
Due to $\|y_i\|_{H^1((0,1))}\leq\|y\|_{H^1((0,1),\ell^2)}$ for any $y\in H^1((0,1),\ell^2)$,
extracting the $i$th component is a continuous linear operator from $H^1((0,1),\ell^2)$ to $H^1((0,1))$.
Therefore, for all $i$ we have $x_i^n\rightharpoonup x_i$ weakly in $H^1((0,1))$ with fixed $x_i^n(0),x_i^n(1)>0$.
Introducing the map $e:H^1((0,1))\to\bR$, $e(z)=\int_0^1(1 + (z(t)+1)^{-2})\dot z(t)^2\,\mathrm d t=\|\dot z\|_{L^2}^2+\|\partial_t\log (z+1)\|_{L^2}^2$,
the desired weak lower semicontinuity now follows via
\begin{align}
\liminf_{n\to\infty}E(x^n)
=\liminf_{n\to\infty}\int_0^1G_{x^n(t)}(\dot x^n(t),\dot x^n(t))\,\mathrm d t
&\geq\sum_{i\in \mathbb{N}}\liminf_{n\to\infty}e(x_i^n)\\
\geq\sum_{i\in \mathbb{N}}e(x_i)
&=\int_0^1G_{x(t)}(\dot x(t),\dot x(t))\,\mathrm d t
=E(x),
\end{align}
using the weak lower semicontinuity of $e$ along sequences $z^n$ in $H^1((0,1))$ with fixed end points $z^n(0),z^n(1)>-1$:
Indeed, let $\liminf_n e(z^n)<\infty$ and assume without loss of generality that the liminf actually is a limit (else pass to a subsequence).
By Poincar\'e's inequality it follows that $\log (z^n+1)$ is uniformly bounded in $H^1((0,1))$,
thus upon extracting a subsequence we have weak convergence $\log (z^n+1)\rightharpoonup w$ in $H^1((0,1))$.
By the compact embedding $H^1((0,1))\hookrightarrow C^0([0,1])$ we additionally have $z^n\to z$ and $\log(z^n+1)\to w$ in $C^0([0,1])$, which implies $w=\log(z+1)$.
Now $\liminf_n e(z^n)\geq e(z)$ follows from the lower semicontinuity of the $L^2$-norm under weak $L^2$-convergence.
\end{example}

\section{Spaces of Immersions and Notation}\label{sec:spaces}
From here on let $M$ be a smooth two-dimensional compact manifold without boundary. Furthermore, we will equip $M$ with a fixed smooth Riemannian metric~$\ol g$.
\subsection{Sobolev Spaces}\label{sec:SobolevSpaces}
Here we will introduce Sobolev spaces on sections of natural vector bundles over $M$, where we follow the classic textbook by Triebel~\cite[Section~7]{triebel1992theory2}.
For $l\in \mathbb N$ and $p\geq 1$ we write $W^{l,p}(\mathbb R^m,\mathbb R^n)$ for the Sobolev space of $\mathbb R^n$-valued functions on $\mathbb R^m$, i.e., the space of all functions that have weak derivatives up to order $l$ in $L^p$. In the special case $p=2$, we use the notation
$H^{l}(\mathbb R^m,\mathbb R^n):=W^{l,2}(\mathbb R^m,\mathbb R^n)$.

Next we consider a  a vector bundle $E$ of rank $n\in\mathbb N$ over $M$, and we let $C^{\infty}(M,E)$ denote the corresponding space of (smooth) sections. To generalize Sobolev spaces to this setting we choose a finite vector bundle atlas and a subordinate partition of unity, i.e.,
we let $(u_i:U_i \to u_i(U_i)\subseteq \mathbb R^m)_{i\in I}$ be a finite atlas for $M$ with vector bundle charts $(\psi_i:E|U_i \to U_i\times \mathbb R^n)_{i\in I}$
and $(\varphi_i)_{i\in I}$ a smooth partition of unity subordinate to $(U_i)_{i \in I}$.
Note that we can choose open sets $U_i^\circ$ such that $\on{supp}(\varphi_i)\subset U_i^\circ\subset \overline{U_i^\circ}\subset U_i$  and each $u_i(U_i^\circ)$ is an open set in $\mathbb R^m$ with Lipschitz boundary.

Using this we define for each $l \in \mathbb N$, $p\geq 1$ and $f \in C^{\infty}(M,E)$ the Sobolev norm of $f$ via
\begin{equation}
|||f|||_{W^{l,p}}^2 := \sum_{i \in I} \|\on{pr}_{\mathbb R^n}\circ\, \psi_i\circ (\varphi_i \cdot f)\circ u_i^{-1} \|_{W^{l,p}(\mathbb R^m,\mathbb R^n)}^2,
\end{equation}
where $\on{pr}_{\mathbb R^n}$ denotes the projection onto $\mathbb R^n$.
Then $|||\cdot|||_{W^{l,p}}$ is a norm, and we write $W^{l,p}(M,E)$ for the Banach completion of $C^{\infty}(M,E)$ under this norm. Furthermore, the space $W^{l,p}(M,E)$ is independent of the choice of atlas and partition of unity, up to equivalence of norms. In the special case $p=2$ we write $W^{l,p}(M,E)=H^{l}(M,E)$. We refer to \cite[Section~7]{triebel1992theory2} and \cite[Section~6.2]{grosse2013sobolev} for further details.

\subsection{The Manifold of Immersions}
For $d\geq 3$ and $l > 2$ we define the space of Sobolev immersions from $M$ into $\bR^d$ via
\begin{equation}
\Imm^l(M,\bR^d):=\left\{f\in H^l(M,\bR^d): T_xf:T_xM\to\bR^d \text{ is injective for all }x\in M\right\},
\end{equation}
where $H^l(M,\bR^d)$ denotes the space of Sobolev functions on $M$ with values in $\bR^d$. Note that $H^l(M,\bR^d)$ embeds in $C^1(M,\bR^d)$ precisely if $l>2$, and thus the (pointwise) condition of $f$ being an immersion (i.e.\ $Tf$ being injective) is well-defined, cf.~\cite{eichhorn2007global}.
The above definition extends to $l=\infty$ leading to the space $\Imm(M,\bR^d):=\Imm^{\infty}(M,\bR^d):=\bigcap_{l=1}^\infty\Imm^l(M,\bR^d)$ of all smooth immersions.

For any finite $l$, the space $\Imm^l(M,\bR^d)$ is an open subset of the Hilbert space $H^l(M,\bR^d)$ and thus is itself an infinite-dimensional Hilbert manifold. For $l=\infty$ it is only a Fr\'echet manifold as the modeling space $H^{\infty}(M,\bR^d)=\bigcap_{l=1}^\infty H^l(M,\bR^d)=C^{\infty}(M,\bR^d)$ is only a Fr\'echet space, but not a Hilbert or Banach space~\cite{kriegl1997convenient}.

\subsection{The Group of Diffeomorphisms and the Shape Space of Unparametrized Surfaces}\label{sec:diffgroup_action}
For $l > 2$ we will also consider the group of Sobolev diffeomorphisms $\on{Diff}^l(M)$, which is a  Hilbert manifold as well, 
since it is an open subset of  $H^l(M,M)$. For any finite $l$ it is not a Lie group, but only a half Lie group~\cite{marquis2018half,bauer2025regularity}: Right translations are smooth, but left translations are merely continuous. Only
for $l=\infty$ left translations become smooth, and it is an infinite-dimensional
Lie group in the sense of \cite[Section~43]{kriegl1997convenient}.

The group $\on{Diff}^l(M)$ acts continuously  on the space $\Imm^l(M,\bR^d)$ by composition from the right~\cite{inci2013diffeo}. The action is given by the mapping
\begin{equation}
r:\begin{cases}
       \Imm^l(M,\bR^d) \times \Diff^l(M) &\to \Imm^l(M,\bR^d),\\(f,\varphi) &\mapsto  f \circ \varphi,
       \end{cases}
\end{equation}
where the tangent prolongation of this group action is given by the mapping (using the same symbol for simplicity)
\begin{align}
r:
\begin{cases}T\Imm^l(M,\bR^d) \times \Diff^l(M) &\to T\Imm^l(M,\bR^d), \\
(f,h,\varphi) &\mapsto (f\circ\varphi,h \circ \varphi).
  \end{cases}
\end{align}
This allows us to define the quotient space of  unparametrized surfaces  via
\begin{equation}
\Shape^l(M,\bR^d):=\Imm^l(M,\bR^d)/\Diff^l(M).
\end{equation} 
This space does not carry the structure of an infinite-dimensional manifold; for $l=\infty$ it is an infinite-dimensional orbifold (thus almost a manifold)~\cite{cervera1991action}, whereas for finite $l$ we only know that it is a Hausdorff topological space~\cite[Proposition 6.2]{bruveris2015completeness}.

\subsection{The Induced Geometry of a Surface $f$}\label{sec:inducedGeometry}
Any surface $f \in \Imm^l(M,\bR^d)$ induces a Riemannian metric $g=f^*\langle \cdot,\cdot\rangle_{\bR^d}$
on $M$, i.e.,
\begin{align}
g(X,Y)= \langle Tf.X,Tf.Y\rangle_{\bR^d},
\end{align}
where $X,Y$ are vector fields on $M$.
Note that, for $l\neq \infty$, the Riemannian metric $g$ is not smooth either, but only of regularity $H^{l-1}$ (recall that $H^{l-1}$ for $l>2$ forms an algebra so that products of $H^{l-1}$-functions have same regularity).
It is at least continuous, though, as we assumed $l>2$ so that $H^{l-1}$ embeds into $C^0$ by the Sobolev embedding theorem.
As is customary, we will sometimes identify $g$ with its associated linear operator (more precisely, its associated bundle homomorphism) $TM\to T^*M$ using the same symbol.
Furthermore, the Riemannian metrics $g$ and $\langle\cdot,\cdot\rangle_{\bR^d}$ induce fiber metrics $g^{i,m}_j$ on all tensor bundles $(TM)^{\otimes i}\otimes (T^*M)^{\otimes j}\otimes (\bR^d)^{\otimes m}$. To keep the notation less cumbersome we will suppress the tensor type in the notation for the fiber metric, i.e., we will write $g=g^{i,m}_j$ independently of the type of tensor it acts on.
By identifying the space $L(P,Q)$ of bundle homomorphisms between tensor bundles $P$ and $Q$ with $Q\otimes P^*$ for $P^*$ the dual bundle to $P$, the metric also extends to such bundle homomorphisms (such as $g$ itself).

\begin{remark}[Fiber Metric for Elementary Tensors]\label{rem:fiberMetric}
The induced fiber metric is fully specified by its value for elementary tensors. Consider two elementary tensors $\alpha_1$ and $\alpha_2$ with
\begin{equation}\alpha_n=X^1_n\otimes \ldots \otimes X^i_n\otimes Y^1_n\otimes \ldots \otimes Y^j_n\otimes Z^1_n\otimes \ldots \otimes Z^m_n,\qquad n\in\{1,2\},
\end{equation}
where $X^1_n,\ldots, X^i_n\in T_xM$,
$Y^1_n,\ldots,Y^j_n\in T_x^*M$,
and $Z^1_n,\ldots,Z^m_n\in\bR^d$.
Then the induced fiber metric is defined as
\begin{equation}
g(\alpha_1,\alpha_2)=\prod_{n=1}^i g(X^n_1,X^n_2) \prod_{n=1}^j g^{-1}(Y^n_1,Y^n_2)\prod_{n=1}^m\langle Z^n_1,Z^n_2\rangle_{\bR^d},
\end{equation}
where $g^{-1}(Y_1,Y_2)\coloneq g(X_1,X_2)$ with $X_n$ defined as $g(X_n,\cdot)=Y_n$, $n=1,2$.
\end{remark}

\begin{remark}[Fiber Metric for Bundle Homomorphisms]\label{rem:metricBundleHom}
For bundle homomorphisms $A,B\in L(TM,T^*M)$ we can calculate $g(A,B)=\Tr(g^{-1}A^*g^{-1}B)$.
Indeed, let $X_1,X_2\in T_xM$ be a $g$-orthonormal basis with dual basis $X_1^*,X_2^*\in T_x^*M$, then $A$ and $B$ are identified with $A(X_1)\otimes X_1^*+A(X_2)\otimes X_2^*$,  $B(X_1)\otimes X_1^*+B(X_2)\otimes X_2^*$, thus
\begin{align}
\Tr(g^{-1}A^*g^{-1}B)
&=\Tr\left(\sum_{i=1}^2g^{-1}(X_i^*)\otimes A(X_i)\sum_{j=1}^2 g^{-1}(B(X_j))\otimes X_j^*\right)\\
&=\sum_{i,j=1}^2\langle g^{-1}(X_i^*),X_j^*\rangle\langle A(X_i),g^{-1}(B(X_j))\rangle
=g(A,B)
\end{align}
Analogously, $g(A,B)=\Tr(gA^*gB)$ for $A,B\in L(T^*M,TM)$ and $g(A,B)=\Tr(g^{-1}A^*gB)$ for $A,B\in L(TM,TM)$ as well as $g(A,B)=\Tr(gA^*g^{-1}B)$ for $A,B\in L(T^*M,T^*M)$.
Obviously, in all cases $|A|_g=|A^*|_g$.
Furthermore, for $A\in L(T^*M,TM)$ and $B\in L(TM,T^*M)$ we have
\begin{align}
|AB|_g^2
&=\Tr(g^{-1}B^*A^*gAB)
=\Tr(A^*gAgg^{-1}Bg^{-1}B^*)\\
&\leq\Tr(A^*gAg)\Tr(g^{-1}Bg^{-1}B^*)
=|A|_g^2|B|_g^2
\end{align}
due to the submultiplicative property of the trace for symmetric positive semidefinite (spd) operators
(note that $A^*gAg$ and $g^{-1}Bg^{-1}B^*$ are both spd with respect to the inner product $g$).
Analogously, $|BA|_g\leq|B|_g|A|_g$ as well as $|AY|_g\leq|A|_g|Y|_g$, $|BX|_g\leq|B|_g|X|_g$ for $X\in TM$, $Y\in T^*M$.
\end{remark}

We denote the induced volume form of $f$ ($g$, resp.) by $\vol$ and the associated total volume  by $\Vol:=\int_{M}\vol$.
Finally, we denote the corresponding Levi-Civita connection on $M$ by~$\nabla$.
The connection extends from $TM$ to the full fiber bundle by the Leibniz rule.
In particular, the connection acting on general $(i,j,m)$-tensor fields with $m>0$ is obtained by combining it with the trivial connection on $\bR^d$, cf.~\cite[Section 3.7]{bauer2011sobolev} for a detailed exposition.
This also allows us to define the induced connection Laplacian or rather Bochner Laplacian by $\Delta=\nabla^*\nabla$, where $\nabla^*$ denotes the adjoint of the covariant derivative with respect to the $L^2(g)$-inner product.
Note that $\nabla h\in T^*M$ for a scalar function $h$ on $M$, i.e., $\nabla h$ does not denote the gradient of $h$ (which would be the tangent vector associated to $\nabla h$ by the Riesz isomorphism).

Next we consider the normal bundle $\on{Nor}(f)$ of an immersion $f$ as a sub-bundle of $f^*T\bR^d$
whose fibers consist of all vectors that are orthogonal to the image of $f$, i.e.,
\begin{equation}
 \on{Nor}(f)_x = \big\{ Y \in T_{f(x)}\bR^d : \forall X \in T_xM : \langle Y,T_xf.X\rangle_{\bR^d}=0  \big\}.
\end{equation}
Note that any vector field $h$ along $f$ can be decomposed uniquely
into parts {\it tangential} and {\it normal} to $f$ as
$$h=Tf.h^\top + h^\bot,$$ 
where $h^\top$ is a vector field on $M$ and $h^\bot$ is a section of the normal bundle $\on{Nor}(f)$. Using the decomposition in tangential and normal parts we  have the following formula for the covariant derivative:
For $X,Y$ tangent vector fields to $M$,
\begin{equation} 
\nabla_X(Tf.Y)=Tf.(\nabla_X(Tf.Y))^\top + (\nabla_X(Tf.Y))^\bot = Tf.\nabla_X Y + S(X,Y),
\end{equation}
where $\nabla_X(Tf.Y)=\nabla_{Tf.X}(Tf.Y)$ (with $\nabla$ on the left-hand side the connection from \cite[Section 3.7]{bauer2011sobolev} and $\nabla$ on the right-hand side the standard Euclidean connection) and $S$ is the \emph{second fundamental form of $f$}.
It is a symmetric bilinear form with values in the normal bundle of $f$. 
When $Tf$ is seen as a section of $T^*M \otimes f^*T\bR^d$, one has $S=\nabla Tf$ since
\begin{equation}
S(X,Y) = \nabla_X(Tf.Y) - Tf.\nabla_X Y = (\nabla Tf)(X,Y).\end{equation}
Taking the trace of $S$ with respect to $g$ (the trace of the operator $S$ induced by $S(X,Y)=g(X,SY)$) yields the \emph{vector-valued mean curvature} \begin{equation}
   \H=\Tr^g(S) 
 \in  H^{l-2}\big(M,\on{Nor}(f)\big),
\end{equation}
whose (pointwise Euclidean) norm $|H|$ is the scalar mean curvature.

Recall, that we assumed that $M$ is equipped in addition with a fixed  (auxiliary) Riemannian metric $\ol{g}$ on $M$; fixed refers to the fact that $\ol{g}$ doe not depend on $f$. We denote the corresponding connection by $\ol\nabla$ and the volume form $\ol\vol$.
We will frequently work with  the Radon–Nikodym derivative of $\vol$ with respect to $\ol\vol$, which we will denote via
\begin{equation}
    \rho=\vol/\ol\vol\in H^{l-1}(M,\bR_{>0}).
\end{equation}
We will also need the difference between the covariant derivative of the surface metric $g$ and the background metric $\ol{g}$. To this end, we define the tensor
\begin{equation}\label{eq:Gamma}
\Gamma = \nabla - \ol\nabla.
\end{equation}
The type of this tensor depends on the type of covariant derivative (what type of tensor it differentiates). Generally, when differenting sections of a vector bundle $E$ over $M$, we have that $\Gamma$ is a section of $E^*\otimes T^*M\otimes E$.
Specific examples are as follows:
\begin{itemize}
\item
If applied to a $(0,0,0)$-tensor (a scalar), $\Gamma=0\in T^*M$, since derivatives of functions $h$ are the same with respect to all connections, $\nabla h=\ol \nabla h=Th$.
\item
If applied to a $(1,0,0)$-tensor (a tangent vector), $\Gamma$ in local coordinates equals the tensor of the difference in Christoffel symbols of the second kind, which form a $(1,2,0)$-tensor. 
\item
If applied to a $(0,1,0)$-tensor (a cotangent vector), $\Gamma$ in local coordinates again equals the tensor of the difference in Christoffel symbols of the second kind
(compared to the previous case there is a sign change due to $\nabla_j \partial_i = \Gamma_{ij}^k \partial_k$ versus $\nabla_j (dx^i) = -\Gamma^i_{jk} dx^k$ for $\Gamma_{ij}^k$ the Christoffel symbols of the second kind), which form a $(1,2,0)$-tensor.
\item
Denoting the previous two instances by $\ChristVec $ and $\ChristCovec $ (indicating $\Gamma$ acting on tangent and cotangent vectors, respectively),
$\Gamma$ applied to a $(i,j,m)$-tensor $\alpha=X^1\otimes \ldots \otimes X^i\otimes Y^1\otimes \ldots \otimes Y^j\otimes Z^1\otimes \ldots \otimes Z^m$ by the Leibniz rule reads
\begin{equation}\label{eqn:GammaTensor}
\Gamma\alpha
=\sum_{n=1}^i\alpha^n+\sum_{n=1}^j\alpha_n,
\end{equation}
where $\alpha^n$ is $\alpha$ with the factor $X^n$ replaced by $\ChristVec X^n$
and $\alpha_n$ is $\alpha$ with the factor $Y^n$ replaced by $\ChristCovec Y^n$.
Thus $\Gamma$ is a $(i+j,i+j+1,2m)$-tensor
or more precisely a section of $(TM)^{\otimes(i+j)}\otimes(T^*M)^{\otimes(i+j+1)}\otimes(\bR^d)^{\otimes 2m}$.
\end{itemize}
Christoffel symbols of the first kind are just second derivatives of the metric and thus have $H^{l-2}$-regularity;
Christoffel symbols of the second kind are obtained from those of the first kind by composition from the left with $g$ and $\ol g$, respectively,
so that they have the same regularity (recall that $g$ has regularity $H^{l-1}$, which forms an algebra).
Consequently, $\ChristVec,\ChristCovec$ are both of $H^{l-2}$-regularity, and thus so is $\Gamma$, independent of its tensor type.

{\bf The dependence on $f$ of all of the objects introduced in this section should be kept in mind in the sequel.}

\subsection{Sobolev (Semi-)Norms}\label{sec:SobolevNorms}
Next we introduce several (semi)-norms on spaces of $(i,j,m)$-tensors. In the following let $g=f^*\langle \cdot,\cdot\rangle_{\bR^d}$ as in \cref{sec:inducedGeometry}, let $l\geq 0$ and let $p\geq 1$. For sake of simplicity of the presentation we assume that $f\in \Imm^{\infty}(M,\bR^d)$; otherwise we need to restrict $l$ to be not too big depending on the regularity of the immersion $f$ and $p$.
We denote the $W^{l,p}(g)$-norm of a $(i,j,m)$-tensor field $h$ by $\|h\|_{W^{l,p}(g)}$, i.e.,
\begin{equation}
\|h\|^p_{W^{l,p}(g)}= \int_M g(h,h)^{p/2}\vol+\int_M g(\nabla^l h,\nabla^l h)^{p/2}\vol.
\end{equation}
We will also need the homogeneous semi-norm of the same order given by
\begin{equation}
\|h\|^p_{\dot W^{l,p}(g)}= \int_M g(\nabla^l h,\nabla^l h)^{p/2}\vol.
\end{equation}
For the special case $p=2$ we also write
\begin{equation}
\|h\|_{W^{l,2}(g)}=\|h\|_{H^{l}(g)},\qquad \|h\|_{\dot W^{l,2}(g)}=\|h\|_{\dot H^{l}(g)},
\end{equation}
and for $l=0$ we obtain (up to a constant) the corresponding $L^p$-norms, which we denote as $\|h\|_{L^{p}(g)}$.

We will also use the analogous norms calculated with respect to the fixed metric $\bar g$ for which we will suppress the dependence on the Riemannian metric, i.e.,
\begin{equation}
\|h\|_{W^{l,p}(\ol{g})}=\|h\|_{W^{l,p}}\;
\end{equation}
and similarly for the other norms introduced above. Finally, we note that these norms lead to the same Sobolev spaces as defined in \cref{sec:SobolevSpaces}. For the smooth Riemannian metric $\ol{g}$ the norm equivalence $|||\cdot|||_{W^{l,p}}\sim\|\cdot\|_{W^{l,p}}$ is well known, see e.g.~\cite{triebel2006theory}. For metrics of finite regularity, such as $g=f^*\langle \cdot,\cdot\rangle_{\bR^d}$ with $f\in \Imm^{k}(M,\bR^d)$, one has to restrict $l$ in terms of $k$. In the following Lemma we show this result in the case $p=2$.
\begin{lemma}[Sobolev regularity of derivatives]\label{thm:derivativeRegularity}
Let $k\geq l\geq 3$ and $f\in \Imm^k(M,\bR^d)$, $h\in T_f \Imm^k(M,\bR^d)$, then $\nabla^k h$ has $H^{k-l}$-regularity.
\end{lemma}

\begin{proof}
For $l\in\{0,1\}$ this follows from $\nabla^lh=\ol\nabla^lh$.\
For larger $l$ it follows by induction from $\nabla^lh=(\ol\nabla+\Gamma)\nabla^{l-1}h$, which lies in $H^{k-l}$ due to $\nabla^{l-1}h$ being $H^{k-l+1}$- and $\Gamma$ beging $H^{k-2}$-regular:
Indeed, for $k\geq4$ or for $k=3$ and $l=2$, their product has Sobolev-regularity $\min\{k-l+1,k-2\}\geq k-l$,
while for $l=k=3$ both functions lie in $L^4$ via Sobolev-embedding so that their product lies in $L^2$.
\end{proof}

\subsection{Riemannian Metrics on Spaces of Immersions}
In this article we are interested in completeness properties of reparametrization-invariant Riemannian metrics $G$ on the space of immersions, i.e., Riemannian metrics $G$ such that
\begin{align}
G_f(h_1,h_2)=G_{f\circ\varphi}(h_1\circ\varphi,h_2\circ\varphi)
\end{align}
for all  $f\in\Imm^l(M,\bR^d)$, $h_1,h_2\in T_f \Imm^l(M,\bR^d)$, and $\varphi\in \Diff^l(M)$.

For the proof of our main theorems we will also consider a non-invariant auxiliary Riemannian metric, namely the $H^k$-metric induced by the background metric $\ol{g}$.
\begin{definition}[Background $H^k$-Metric]\label{def:backgroundmetric}
For any $l\geq k\geq 3$ we consider the footpoint-independent, non-reparametrization-invariant $H^k$-metric on both $\Imm^l(M,\bR^d)$ and on $H^l(M,\bR^d)$ given via
\begin{align}\label{equation:backgroundmetric}
\ol G^k_f(h_1,h_2) 
&= 
\int_{M} \ol g(h_1,h_2)\ol \vol 
+
\int_{M}\ol g(\ol \nabla^k h_1,\ol\nabla^k h_2)\ol \vol,
\end{align} 
where $\ol \nabla$ and $\ol \vol$ denote the covariant derivative and volume form of $\ol{g}$.
\end{definition}

\section{Geometric Sobolev Embedding Theorems}\label{sec:SobolevEmbeddings}
In the following we will formulate several Sobolev type estimates, where we will pay careful attention to the dependence of the constants on the immersion $f$. Several of the below estimates are formulated for smooth immersions and smooth functions in the cited references, whereas we will only require the immersion to be of sufficient Sobolev regularity. We will justify this generalization at the end of this \namecref{sec:SobolevEmbeddings}, where we present an approximation result in~\cref{lem:approximation_MSSestimates}

We start by formulating the so-called Michael--Simon--Sobolev inequality.
\begin{theorem}[{Michael--Simon--Sobolev Inequality~\cite[Theorem 2.1]{michael1973sobolev}}]\label{thm:MSSineq}
Let $f\in \on{Imm}^3(M,\bR^d)$, let $g=f^*\langle\cdot,\cdot\rangle_{\bR^d}$, and let $h$ be a $W^{1,1}$-section of the $(i,j,m)$-tensor bundle. Then
\begin{align}\label{eqn:MichaelSimon}
 \|h\|_{L^2(g)}\leq C\left( \|\nabla h\|_{L^1(g)}+\||\H|h\|_{L^1(g)}\right),
\end{align}
where $\H$ denotes the mean curvature of $f$ and where $C$ is independent of $f$.
\end{theorem}
\begin{proof}
The result is a classic result for functions $\hat h\in W^{1,1}(M,\mathbb R)$, see e.g.~\cite[Theorem 2.1]{michael1973sobolev}. To see that it actually also holds in the setting of $(i,j,m)$-tensors  we let $\hat h=\sqrt{g(h,h)}=|h|_g$, and use that $\|\hat h\|_{L^p(g)}=\|h\|_{L^p(g)}$ and $\|\nabla \hat h\|_{L^p(g)}\leq\|\nabla h\|_{L^p(g)}$ due to $\nabla \hat h=g(h/|h|_g,\nabla h)$.
\end{proof}

As a direct corollary we obtain the following result bounding the $L^4$-norm in terms of the (mean-curvature-weighted) $H^1$-norm.
\begin{corollary}[$L^4$-$H^1$ Sobolev Inequality]\label{thm:L4toH1}
Let $f\in \on{Imm}^3(M,\bR^d)$, let $g=f^*\langle\cdot,\cdot\rangle_{\bR^d}$, and let $h$ be a $H^1$-section of the $(i,j,m)$-tensor bundle. 
Then
\begin{equation}
\|h\|_{L^4(g)}\leq C\left(\|h\|_{L^2(g)}+\|\sqrt{|\H|}h\|_{L^2(g)}+\|\nabla h\|_{L^2(g)} \right)
\end{equation}
where $\H$ denotes the mean curvature of $f$ and where $C$ is independent of $f$.
\end{corollary}
\begin{proof}
We calculate
\begin{align}
\|h\|^4_{L^4(g)}=\|g(h,h)\|^2_{L^2(g)}
&\lesssim \left(\|\nabla g(h,h)\|_{L^1(g)}+\||\H|g(h,h)\|_{L^1(g)}\right)^2\\
&=\left(2\|g(h,\nabla h)\|_{L^1(g)}+\||\H|g(h,h)\|_{L^1(g)}\right)^2\\
&\leq \left(2\|h\|_{L^2(g)}\|\nabla h\|_{L^2(g)}+\||\H|g(h,h)\|_{L^1(g)}\right)^2\\
&=\left(2\|h\|_{L^2(g)}\|\nabla h \|_{L^2(g)} +\|\sqrt{|\H|}h\|^2_{L^2(g)}\right)^2\\
&\leq \left(\|h\|^2_{L^2(g)}+\|\nabla h\|^2_{L^2(g)}+\|\sqrt{|\H|}h\|^2_{L^2(g)}\right)^2\\
&\leq \left(\|h\|_{L^2(g)}+\|\nabla h\|_{L^2(g)}+\|\sqrt{|\H|}h\|_{L^2(g)}\right)^4,
\end{align}
where the first inequality is an application of \cref{thm:MSSineq}, the second inequality follows by the Cauchy--Schwarz inequality, the third inequality is an application of Young's inequality, and the final inequality follows from the fact that
for positive $a,b,c$ one has
$(a^2+b^2+c^2)\leq (a+b+c)^2$.
\end{proof}
Applying the Michael--Simon--Sobolev inequality from \cref{thm:MSSineq} to powers of $h$ and exploiting H\"older's inequality, one can iterate the resulting estimate to obtain the following Sobolev embedding theorem.
\begin{theorem}[{Multiplicative Sobolev Embedding Theorem~\cite[Theorem 5.6]{kuwert2002gradient}}]
\label{thm:Sobembedding}
Let $f\in \on{Imm}^3(M,\bR^d)$, let $2<p\leq\infty$ and let $h$ 
be a $W^{1,p}$-section of the $(i,j,m)$-tensor bundle. Let
$1\leq m\leq\infty$ and $0<\alpha< 1$ with 
$\tfrac1\alpha=(\tfrac12-\tfrac1p)m+1$.  
Then we have 
\begin{equation}\label{eqn:Sobembedding}
\|h\|_{L^{\infty}(g)}\leq C 
\|h\|^{1-\alpha}_{L^{m}(g)}\left(
\|\nabla h\|_{L^{p}(g)}+\||\H|h\|_{L^{p}(g)}
\right)^{\alpha},
\end{equation}
where $\H$ denotes the mean curvature of $f$ and where $C$ is independent of $f$. 
\end{theorem}

In our analysis we will need in addition a different variant of the the Sobolev embedding inequality, for which we first recall the following Sobolev interpolation estimate, which is due to Hamilton.
\begin{theorem}[{Sobolev Interpolation Inequality \cite[Theorem 12.1]{hamilton1982three}}]\label{thm:Sobinterpolation}
Let $f\in \on{Imm}^3(M,\bR^d)$, let $g=f^*\langle\cdot,\cdot\rangle_{\bR^d}$, let $p\geq 1$ and let $h$ be a $W^{2,p}$-section of the $(i,j,m)$-tensor bundle. Let $q,r\geq 1$ with
\begin{equation}
\frac{1}p+\frac1q=\frac1r.  
\end{equation}
Then
\begin{equation}\label{eqn:Sobinterpolation}
\|\nabla h\|^2_{L^{2r}(g)}\leq 
C \|\nabla^2 h\|_{L^{p}(g)}\| h\|_{L^{q}(g)},
\end{equation}
where the constant $C$ does not depend on the immersion $f$.
\end{theorem}
We are now able to formulate the main Sobolev inequality for this article. 
\begin{corollary}[Sobolev Embedding Theorem]
\label{cor:Sobembedding}
Let $f\in \on{Imm}^3(M,\bR^d)$, let $g=f^*\langle\cdot,\cdot\rangle_{\bR^d}$, let $q>1$ and let $h$ be a $W^{2,q}$-section of the $(i,j,m)$-tensor bundle. Then we have
\begin{align}\label{eq:Michael_Simon_L_infty}
\| h\|_{L^{\infty}(g)} \leq C \Vol^{\frac{q-1}{q}} \left( \|\nabla^2h\|_{L^q(g)} + \||\H|^2 h\|_{L^q(g)}
 \right),
\end{align}
where $\H$ denotes the mean curvature of $f$, $\Vol$ the total volume of $f$ and where $C$ is independent of $f$. 
\end{corollary}
\begin{proof}
For the proof we follow the arguments in~\cite[Formula (4.10)]{kuwert2002gradient}: Using \cref{thm:Sobembedding} with $p=2q>2$, $\alpha=\frac12$ and thus $m=\frac{2p}{p-2}$ we obtain that
\begin{align}
\|h\|_{L^{\infty}(g)}&\leq C 
\|h\|^{\frac12}_{L^{\frac{2p}{p-2}}(g)}\left(
\|\nabla h\|_{L^{p}(g)}+\||\H|h\|_{L^{p}(g)}\right)^{\frac12}\\
&\leq C 
\|h\|^{\frac12}_{L^{\infty}(g)} \Vol^{\frac{p-2}{4p}}\left(
\|\nabla h\|_{L^{p}(g)}+\||\H|h\|_{L^{p}(g)}\right)^{\frac12}.
\end{align}
Next we divide both sides by
$\|h\|^{\frac12}_{L^{\infty}(g)}$ and square leading to
\begin{equation}\label{estimate1}
\|h\|_{L^{\infty}(g)}
\leq C^2  \Vol^{\frac{p-2}{2p}}\left(
\|\nabla h\|_{L^{p}(g)}+\||\H|h\|_{L^{p}(g)}\right)
\end{equation}
We then use for the first term the interpolation estimate from \cref{thm:Sobinterpolation} (corresponding to $r=p$ and $q=\infty$) to obtain
\begin{equation}
\|\nabla h\|_{L^{2q}(g)}\leq C_1 \|\nabla^2 h\|^{\frac12}_{L^{q}(g)}\| h\|^{\frac12}_{L^{\infty}(g)}\;.
\end{equation}
Plugging this back into~\eqref{estimate1} yields
\begin{align}
\|h\|_{L^{\infty}(g)}&\leq C_2 \Vol^{\frac{q-1}{2q}}\left(
\|\nabla^2 h\|^{\frac12}_{L^{q}(g)}\| h\|^{\frac12}_{L^{\infty}(g)}+\||\H|h\|_{L^{2q}(g)}
\right)\\
&\leq C_2 \Vol^{\frac{q-1}{2q}}\left(
\|\nabla^2 h\|^{\frac12}_{L^{q}(g)}\| h\|^{\frac12}_{L^{\infty}(g)}+\| h\|^{\frac12}_{L^{\infty}}\||\H|^2h\|^{\frac12}_{L^{q}(g)}
\right),
\end{align}
which yields the desired estimate after dividing both sides by $\| h\|^{\frac12}_{L^{\infty}}$, squaring, and applying Young's inequality.
\end{proof}

Next we will collect a rather simple estimate, that allows us to control $\dot H^q$-norms by the $H^{q'}$-norm for $q'\geq q$.
\begin{lemma}[Interpolation Inequality]\label{lem:metricdom}
Let $l,q,q'\in \bN$ with $l\geq q'\geq q$ and $l\geq 3$. Let $f\in \on{Imm}^l(M,\bR^d)$ and let $h$ be a $H^{q'}$-section of the $(i,j,m)$-tensor bundle.  Then
\begin{equation}
\| h\|^2_{\dot H^{q}(g)} \lesssim  \| h\|^2_{L^2(g)}+\| h\|^2_{\dot H^{q'}(g)}.
\end{equation}
\end{lemma}

\begin{proof}
For $q=1$ and $q'=2$ the result is obtained via
\begin{align}
&\| h\|^2_{\dot H^{1}(g)}= \int_M g(\nabla h,\nabla h)\,\vol = \int_M g( h,\nabla^*\nabla h)\,\vol
\\&\qquad\leq \sqrt{\int_M g(h,h)\,\vol}\sqrt{\int_M g(\nabla^*\nabla h,\nabla^*\nabla h)\,\vol}
\\&\qquad\lesssim \| h\|_{L^2(g)}\| h\|_{\dot H^{2}(g)}
\leq\tfrac1{2\epsilon}\| h\|^2_{L^2(g)}+\tfrac\epsilon2\| h\|^2_{\dot H^{2}(g)},
\end{align}
where the last inequality follows for arbitrary $\epsilon>0$ by Young's inequality.
The second to last inequality follows from the equality between the negative Bochner Laplacian $-\Delta=-\nabla^*\nabla$ and the connection Laplacian $\Tr^g(\nabla^2)$, which implies the pointwise inequality
\begin{equation}
g(\nabla^*\nabla h,\nabla^*\nabla h)\leq \operatorname{dim}(M) g(\nabla^2 h, \nabla^2 h)=2 g(\nabla^2 h, \nabla^2 h).
\end{equation}
By induction this implies $\| h\|^2_{\dot H^{q}(g)}\leq C_q\| h\|^2_{L^2(g)}+c_q\| h\|^2_{\dot H^{q+1}(g)}$ for arbitrary $q$, where $c_q>0$ may be chosen arbitrarily small as long as $C_q>0$ is big enough (depending on $q,c_q$).
Indeed, assume this holds for $q<Q$, then we obtain
\begin{align}
\|h\|^2_{\dot H^{Q}(g)}
&=\|\nabla h\|^2_{\dot H^{Q-1}(g)}\\
&\leq C_{Q-1}'\|\nabla h\|^2_{L^2(g)}+c_{Q-1}'\|\nabla h\|^2_{\dot H^{Q}(g)}
=C_{Q-1}'\|h\|^2_{\dot H^1(g)}+c_{Q-1}'\|h\|^2_{\dot H^{Q+1}(g)},
\end{align}
where we denoted the constants $C_{Q-1}',c_{Q-1}'$ since we will later also employ a second pair $C_{Q-1},c_{Q-1}$ of admissible constants.
Now, again inductively, we have
\begin{align}
\|h\|^2_{\dot H^1(g)}
&\leq C_1\|h\|^2_{L^2(g)}+c_1\|h\|^2_{\dot H^2(g)}
\leq C_1\|h\|^2_{L^2(g)}+c_1C_2\|h\|^2_{L^2(g)}+c_1c_2\|h\|^2_{\dot H^3(g)}\\
&\leq\ldots
\leq A\|h\|^2_{L^2(g)}+B\|h\|^2_{\dot H^Q(g)}
\qquad\text{for }
A=\sum_{j=1}^{Q-1}C_j\prod_{i=1}^{j-1}c_i,\,B=\prod_{j=1}^{Q-1}c_j.
\end{align}
Inserting into the above we get
\begin{align}
\|h\|^2_{\dot H^{Q}(g)}
\leq C_{Q-1}'A\|h\|^2_{L^2(g)}+C_{Q-1}'B\|h\|^2_{\dot H^Q(g)}+c_{Q-1}'\|h\|^2_{\dot H^{Q+1}(g)},
\end{align}
which concludes the induction step for $C_Q=C_{Q-1}'A/(1-C_{Q-1}'B)$ and $c_Q=c_{Q-1}'/(1-C_{Q-1}'B)$
(note that $B$ and $c_Q$ can be made arbitrarily small by taking the $c_{Q-1}'$, $c_j$, $j<Q$, sufficiently small).
Finally, let $q,q'>q$ arbitrary, then by the above
\begin{equation}
\| h\|^2_{\dot H^{q}(g)}
\lesssim  \| h\|^2_{L^2(g)}\!+\| h\|^2_{\dot H^{q+1}(g)}
\lesssim  \| h\|^2_{L^2(g)}\!+\| h\|^2_{\dot H^{q+2}(g)}
\lesssim \ldots
\lesssim  \| h\|^2_{L^2(g)}\!+\| h\|^2_{\dot H^{q'}(g)}
\end{equation}
as desired.
\end{proof}
Finally we present the approximation result, which justifies why we could formulate all the above results for immersions  $f$ that have merely finite Sobolev regularity.

\begin{lemma}[Convergence of Covariant Derivatives and Mean Curvature]\label{lem:approximation_MSSestimates}
Let $f_n\in\Imm^l(M,\bR^d)$ be a sequence of immersions with $f_n\to f_\infty$ as $n\to\infty$ in $H^l(M,\bR^d)$ for some immersion $f_\infty\in\Imm^l(M,\bR^d)$ and $l\geq3$.
Denote by $g_n$, $\rho_n$, $\nabla_n$, and $\H_n$ the metric, volume density, connection, and mean curvature associated with $f_n$, $n\in\bN\cup\{\infty\}$
(recall that $g_n$ and $\nabla_n$ are extended to elements and sections of tensor bundles as described in \cref{sec:inducedGeometry}). Then the following hold.
\begin{enumerate}[label=(\alph*)]
\item\label{enm:metricConvergence}
$g_n\to g_\infty$ in $H^{l-1}$ and thus in $C^{l-3,\alpha}$ for all $\alpha<1$.
\item
$\rho_n\to\rho_\infty$ in $H^{l-1}(M,\bR)$ and thus in $C^{l-3,\alpha}(M,\bR)$ for all $\alpha<1$.
\item\label{enm:connectionConvergence}
Let $h_n$, $n\in\bN\cup\{\infty\}$, be sections of the $(i,j,m)$-tensor bundle with $h_n\to h_\infty$ in $H^{k}$ for $1\leq k<l$ (in $W^{k-1,p}$ for $1<k<l$, $p\in[1,\infty)$, respectively), then $\nabla_nh_n\to\nabla_\infty h_\infty$ in $H^{k-1}$ (in $W^{k-2,p}$, respectively).
\item
$\nabla_n^l h_n\to\nabla_\infty^lh_\infty$ in $L^2$ for $h_n\to h_\infty$ in $H^l(M,\bR)$.
\item
$\nabla_n^{l-1} h_n\to\nabla_\infty^{l-1}h_\infty$ in $L^p$ for $h_n\to h_\infty$ in $W^{l-1,p}(M,\bR)$.
\item\label{enm:meanCurvature}
$\H_n\to \H_\infty$ in $H^{l-2}(M,\bR^d)$ and thus in $W^{l-3,p}(M,\bR^d)$ for all $p\in[1,\infty)$.
\end{enumerate}
\end{lemma}
\begin{proof}
\begin{enumerate}[label=(\alph*)]
\item
We have $Tf_n\to Tf_\infty$ in $H^{l-1}(M,\bR^d)$.
Since $H^{l-1}(M,\bR)$ is a Banach algebra for $l\geq3$, this implies $g_n=\langle Tf_n\cdot,Tf_n\cdot\rangle_{\bR^d}\to\langle Tf_\infty\cdot,Tf_\infty\cdot\rangle_{\bR^d}=g_\infty$ in $H^{l-1}$
in the sense that the coordinate representation $G_n$ converges to $G_\infty$ in $H^{l-1}$ in any chart.
By Sobolev embedding this also implies convergence in $C^{l-3,\alpha}$ for all $\alpha<1$ and thus proves the claim for the metric acting on tangent vectors.
As for the metric $g_n$ acting on cotangent vectors, note that its coordinate representation is just the matrix inverse of $G_n$.
Due to $f_\infty\in\Imm^l(M,\bR^d)$ we have that on every chart $G_\infty$ is uniformly bounded away from the set of singular matrices.
Due to $f_n\to f_\infty$ in $H^l$ (and thus in $C^1$ by Sobolev embedding), the same holds true for $G_n$, where the bound is independent of $n$.
Since matrix inversion is smooth away from singular matrices, this implies $G_n^{-1}\to G_\infty^{-1}$ in $H^{l-1}$ in any chart, which thus proves the claim for the metric acting on covectors.
Finally, since $H^{l-1}(M,\bR)$ is a Banach algebra, from the definition of the fiber metric in \cref{rem:fiberMetric} we obtain $g_n\to g_\infty$ in $H^{l-1}$ and thus $C^{l-3,\alpha}$
independent of what tensor bundle section it operates on.
\item
Let $\ol G,G_n$ be the matrix representations of $\ol g,g_n$ in a chart, then $\rho_n=\sqrt{\det G_n/\det\ol G}$ for $n\in\bN\cup\{\infty\}$.
Since $G_n\to G_\infty$ in the Banach algebra $H^{l-1}$, also $\det G_n\to\det G_\infty$ in $H^{l-1}(M,\bR)$.
By the smoothness and uniform positive definiteness of $\ol G$ on any chart we further get $\det G_n/\det\ol G\to\det G_\infty/\det\ol G$ in $H^{l-1}(M,\bR)$.
Finally, since $\det G_n>C$ for some constant $C>0$ independent of $n$ (see previous point) and the square root is smooth on $[C,\infty)$ with uniformly bounded derivatives,
we obtain the desired result.
\item
If we introduce the tensor $\Gamma_n=\nabla_n-\ol\nabla\in H^{l-2}$,
this follows from $\nabla_nh_n=(\ol\nabla+\Gamma_n)h_n$.
Indeed, $\ol\nabla h_n\to\ol\nabla h_\infty$ in $H^{k-1}$ (in $W^{k-2,p}$, respectively) by \cref{sec:SobolevNorms,lem:metricdom}.
As for $\Gamma_n h_n$, first consider the Hilbert space setting.
In the case $l\geq4$, so that $H^{l-2}$ is a Banach algebra,
$\Gamma_nh_n$ converges to $\Gamma_\infty h_\infty$ in $H^{n}$ with $n=\min\{l-2,k\}\geq k-1$ \cite[Lemma 2.3]{inci2013diffeo}, as desired.
For $l=3$ and $k=2$ we again have $\Gamma_nh_n\to\Gamma_\infty h_\infty$ in $H^1$, as desired (this time, $h_n$ lies in the Banach algebra).
Finally, for $l=3$ and $k=1$ we have $\Gamma_nh_n\to\Gamma_\infty h_\infty$ in $L^2$, as desired, since by Sobolev embedding both $\Gamma_n$ and $h_n$ converge in any $L^q$.
Next, consider the non-Hilbert space setting.
Here one can directly check $\Gamma_nh_n\to\Gamma_\infty h_\infty$ in $W^{k-2,p}$, since by Sobolev embedding $\Gamma_n\to\Gamma_\infty$ in $W^{k-2,q}$ for any $q$.
\item
This follows from the fact $\nabla_nh_n=\ol\nabla h_n\to\ol\nabla h_\infty=\nabla_\infty h_\infty$ in $H^{l-1}(M,TM)$ and the previous point.
\item
Again, this follows from $\nabla_nh_n=\ol\nabla h_n\to\ol\nabla h_\infty=\nabla_\infty h_\infty$ in $W^{l-2,p}(M,TM)$ and the previous points.
\item
Denote by $S_n=\nabla_nTf_n$ and $S_\infty=\nabla_\infty Tf_\infty$ the second fundamental form of $f_n$ and $f_\infty$, respectively.
Since $Tf_n\to Tf_\infty$ in $H^{l-1}$, we have $S_n\to S_\infty$ in $H^{l-2}$ by point \eqref{enm:connectionConvergence}.
Together with point \eqref{enm:metricConvergence} and the fact that pointwise multiplication $H^{l-1}\times H^{l-2}\to H^{l-2}$ is continuous for $l\geq3$
we obtain $\H_n=\Tr^{g_n}(S_n)\to\Tr^{g_\infty}(S_\infty)=\H_\infty$ in $H^{l-2}$.
\qedhere
\end{enumerate}
\end{proof}

As a consequence, \eqref{eqn:MichaelSimon}, \eqref{eqn:Sobembedding}, and \eqref{eqn:Sobinterpolation} indeed do not only hold for $f\in\Imm(M,\bR^d)$ and $h\in C^\infty(M,\bR)$:
Let $f_n\to f\in\Imm^3(M,\bR^d)$ in $H^3(M,\bR^d)$ with $f_n$ smooth (and associated $g_n,\rho_n,\H_n,\nabla_n$) and $h_n\to h$ in $W^{1,1}(M,\bR)$ with $h_n$ smooth (which is possible by density of smooth functions), then
\begin{equation}
\rho_nh_n^2\to\rho h^2,\quad
\rho_n\nabla_n h_n\to\rho\nabla h,\quad\text{and}\quad
\rho_n\H_nh_n\to\rho \H h
\quad\text{in }L^1
\end{equation}
by \cref{lem:approximation_MSSestimates} for $l=3$ and Sobolev embedding so that $\|h_n\|_{L^2(g_n)}\to\|h\|_{L^2(g)}$, $\|\nabla_nh_n\|_{L^1(g_n)}\to\|\nabla h\|_{L^1(g)}$, and $\||\H_n|h_n\|_{L^1(g_n)}\to\||\H|h\|_{L^1(g)}$ and thus \eqref{eqn:MichaelSimon} holds.
Likewise, if $h_n\to h$ in $W^{1,p}(M,\bR)$, by an analogous argument we obtain that all norms in \eqref{eqn:Sobembedding} converge,
and if $h_n\to h$ in $W^{2,p}(M,\bR)$, we obtain that all norms in \eqref{eqn:Sobinterpolation} converge (note that $\nabla_nh_n=\ol\nabla h_n\to\ol\nabla h=\nabla h$ in $L^{2r}$ by Sobolev embedding).


\section{Bounding Geometric Quantities on Metric Balls}\label{sec:bounds}
This section contains the main technical estimates which we will need to prove the desired completeness result. 
The basic strategy will be to prove Lipschitz continuity of several geometric functions on metric balls, which will in turn allow us to bound the corresponding quantities. 
Here we recall that a function on a Riemannian manifold is called Lipschitz continuous if it is so with respect to the Riemannian distance.

In order to prove Lipschitz continuity on metric balls, we will repeatedly use the following corollary of Gronwall's inequality (cf.\ \cite[Lemma 5.7]{bauer2023sobolev}).
\begin{lemma}[General Lipschitz Bounds]\label{lem:Gronwall}
Let $(\mathcal{M},\mathfrak{g})$ be a Riemannian manifold, possibly of infinite dimension, and let $F$ be a normed space.
Let $f:\mathcal{M}\to F$ be a $C^1$-function such that for each metric ball $B_r(y)$ in $\mathcal{M}$ there exists a constant $C$ with
\[
\|T_xf.v\|_F \le C(1+\|f(x)\|_F)\|v\|_x \qquad \text{for all } \,\,x\in B_r(y), \,\, v\in T_x\mathcal{M}.
\]
Then $f$ is Lipschitz continuous on every metric ball and in particular bounded on every metric ball.
\end{lemma}

The first result concerns the boundedness of the total volume, which requires the mildest conditions for the metric $G$.
\begin{lemma}[Volume Bound]
\label{sqrtvol}
Let $l\geq 3$ and let $G$ be a smooth Riemannian metric on 
$\Imm^l(M,\bR^d)$
such that
\[
\|h\|_{G_f}
 \gtrsim
\|\nabla h\|_{L^{2}(g)}.
\]
Then the function
\begin{equation}
\sqrt{\Vol}:\left( \Imm^l(M,\bR^d),G\right)\to \bR
\end{equation}
is Lipschitz continuous and thus uniformly bounded on any $G$-metric ball in $\Imm^l(M,\bR^d)$.
\end{lemma}
\begin{proof}
Let $f \in H^1((0,1),\Imm^l(M,\bR^d))$ be a path of immersions with velocity $\dot f$. The first variation of the total volume is given by
\[
\partial_t \Vol = \int_M g(Tf,\nabla \dot f) \vol,
\]
see e.g.~\cite[Section 4]{bauer2012almost}.
Using that
\begin{equation}\label{eqn:normMetric}
|Tf|_{g}^2=\Tr(g^{-1}g)=\dim(M)=2
\end{equation}
we obtain
\begin{align}
\left|\partial_t \sqrt{\Vol}\right|
&=\frac{1}{2 \sqrt{\Vol}}
\left| \int_M g(Tf,\nabla \dot f) \vol\right|\\&
 \leq \frac{1 }{2 \sqrt{\Vol}} \left( \int_M g(\nabla \dot f,\nabla \dot f)\vol \right)^{1/2} \left( \int_M 2\vol\right)^{1/2}\\&=
  \frac{1 }{ \sqrt{2}} \left( \int_M g(\nabla \dot f,\nabla \dot f)\vol \right)^{1/2}
  \leq \frac{1 }{ \sqrt{2}}\|\dot f\|_{G_f},
\end{align}
which concludes the proof via \cref{lem:Gronwall}.
\end{proof}
We note that this result only gives us a control of the volume from above, but not from below. This is not surprising as we would expect to  need  a higher order or a scale invariant metric to control the volume also from below. It turns out that assuming control of the $L^{\infty}$-norm of $\nabla h$ allows us to control several important quantities, which we will study in the next \namecref{lem:Linftyassumption}.
\begin{lemma}[Surface Metric Bound]
\label{lem:Linftyassumption}
Let $l\geq 3$ and let $G$ be a smooth Riemannian metric on 
$\Imm^l(M,\bR^d)$
such that
\[
\|h\|_{G_f}
 \gtrsim
\|\nabla h\|_{L^{\infty}(g)}.
\]
Then the functions
\begin{align}
g &:\left( \Imm^l(M,\bR^d),G\right)\to \left(H^{l-1}(M,T^*M\otimes T^*M), \|\cdot\|_{L^\infty(\ol{g})}\right), \\
g^{-1} &:\left( \Imm^l(M,\bR^d),G\right)\to \left(H^{l-1}(M,TM\otimes TM), \|\cdot\|_{L^\infty(\ol{g})}\right)
\end{align}
are Lipschitz continuous and thus uniformly bounded on any $G$-metric ball in the space $\Imm^l(M,\bR^d)$.
By definition, the same holds true for the induced fiber metric.
\end{lemma}

\begin{remark}[Immersion Property and Equivalence of $L^{p}(g)$ and $L^{p}(\ol{g})$ Norms]\label{rem:volumedens}
\cref{lem:Linftyassumption} immediately implies that the volume density $\rho=\vol/\overline{\vol}$ is uniformly bounded from below and above on $G$-metric balls, i.e., for any $G$-metric ball $B_r(f_0)$ of radius $r>0$ centered at $f_0\in\Imm^l(M,\bR^d)$ there exist constants $C_1,C_2>0$ such that
\begin{equation}
C_1\leq\rho\leq C_2\qquad\text{almost everywhere on $M$.}
\end{equation}
This automatically guarantees that the immersion property is preserved on $G$-metric balls. Furthermore, this allows us to freely pass from $L^{p}(g)$ norms to $L^p(\ol{g})$-norms (and vice versa) in all the following estimates.
\end{remark}

\begin{proof}[Proof of \cref{lem:Linftyassumption}]
To prove this result we start by noting the ``change of norms formulas'' (cf.\ \cref{rem:metricBundleHom})
\[
|\ol{g}|_g^2 = \Tr(g^{-1} \ol{g} g^{-1} \ol{g}) = \Tr(\ol{g} g^{-1} \ol{g} g^{-1}) = |g^{-1}|_{\ol{g}}^2,
\]
and similarly when we swap $g$ and $\ol{g}$.
For $B \in L(TM,T^*M)$ we have with \cref{rem:metricBundleHom}
\begin{equation}\label{eq:norm_change}
\begin{split}
|B|_{\ol{g}}^2 
	&= \Tr(\ol{g}^{-1} B^* \ol{g}^{-1} B) = \Tr(\ol{g}^{-1} g g^{-1} B^* \ol{g}^{-1} g g^{-1} B) \\
	&=\Tr( g (g^{-1} B^* \ol{g}^{-1}) g (g^{-1} B \ol{g}^{-1})) = g((g^{-1} B^* \ol{g}^{-1})^*, g^{-1} B \ol{g}^{-1}) \\
	&\le |g^{-1} B^* \ol{g}^{-1}|_{g}|g^{-1} B \ol{g}^{-1}|_{g} \le |g^{-1} B^*|_{g}|g^{-1} B|_{g} |\ol{g}^{-1}|_g^2 \\
	&= |g^{-1} B^*|_{g}|g^{-1} B|_{g} |g|_{\ol{g}}^2 = |B|_{g}^2 |g|_{\ol{g}}^2,
\end{split}
\end{equation}
where in the last equality we used, for $A=g^{-1} B\in \End(TM)$, the equality
\[
	|A|_g^2 = \Tr (g^{-1} A^* g A) = \Tr(g^{-1} B g^{-1} B) = |B|_g^2.
\]

Now let $f \in H^1((0,1),\Imm^l(M,\bR^d))$ be again a path of immersions with velocity $\dot f$ and choose $B=2\sym \langle\nabla \dot f,Tf\rangle_{\bR^d}$ and $A= g^{-1} B$, where $\langle\cdot,\cdot\rangle_{\bR^d}$ refers to the Euclidean inner product in the codomain of $Tf,\nabla\dot f:TM\to\bR^d$.
We have, using \eqref{eqn:normMetric},
\[
|A|_g = |B|_g = 2|\sym \langle\nabla \dot f,Tf\rangle_{\bR^d}|_g \le 2|\langle\nabla \dot f,Tf\rangle_{\bR^d}|_g \le 2|Tf|_g |\nabla \dot f|_g = \sqrt{8}|\nabla \dot f|_g.
\]

We now recall the following first variations formulas from the proofs of \cite[Section 4]{bauer2012almost},
\begin{equation}\label{eq:variationmetric}
\begin{split}
\pl_t g &= 2\sym \langle\nabla \dot f,Tf\rangle_{\bR^d} = B,\\
\pl_t g^{-1} &= -g^{-1} (\pl_t g) g^{-1}.
\end{split}
\end{equation}
We thus have (cf.\ \cref{rem:metricBundleHom})
\[
\begin{split}
|\pl_t g|_{\ol g} &= |B|_{\ol g} \le |B|_{g} |g|_{\ol{g}} = \sqrt{8}|\nabla \dot f|_g  |g|_{\ol{g}},\\
|\pl_t g^{-1}|_{\ol g}^2 
	& = |g^{-1} B g^{-1}|_{\ol g}^2
  = \Tr(\ol{g} A g^{-1} \ol{g} A g^{-1}) \\
	&= \Tr(g^{-1} (\ol{g} A) g^{-1} (\ol{g} A)) = |\ol{g} A|_g^2\\
  &\le |\ol{g}|_g^2 |A|_g^2 =  |g^{-1}|_{\ol g}^2 |A|_g^2
	=  8|\nabla \dot f|_g^2 |g^{-1}|_{\ol g}^2.
\end{split}
\]
From this the result follows via \cref{lem:Gronwall}.
\end{proof}
Next we aim to estimate derivatives of the metric, i.e.\ we will consider
the tensors $\Gamma$ as defined in~\eqref{eq:Gamma}.
In the next \namecref{lem:W2pplusLinftyassumption} we will bound these quantities and in addition consider bounds on the second fundamental form $S$.
\begin{lemma}[Christoffel Symbol and Second Fundamental Form Bounds]
\label{lem:W2pplusLinftyassumption}
Let $l\geq 3$ and let $G$ be a smooth Riemannian metric on 
$\Imm^l(M,\bR^d)$
such that
\[
\|h\|_{G_f}
 \gtrsim
\| \nabla^2 h \|_{L^p(g)} + \|\nabla h\|_{L^{\infty}(g)}.
\]
for some $p\in [1,\infty]$. Then the functions
\begin{align}
S &:\left( \Imm^l(M,\bR^d),G\right)\to \left(H^{l-2}(M,T^*M^{\otimes 2}\otimes\bR^d), \|\cdot\|_{L^p(\ol{g})}\right),\\
\H &:\left( \Imm^l(M,\bR^d),G\right)\to \left(H^{l-2}(M,\bR^d), \|\cdot\|_{L^p(\ol{g})}\right),\\
\Christ{*} &: \left( \Imm^l(M,\bR^d),G\right)\to \left(H^{l-2}(M, T^*M^{\otimes 2} \otimes TM), \|\cdot\|_{L^p(\ol{g})}\right),\qquad *\in\{\Vec,\Covec\},
\end{align}
are Lipschitz continuous and consequently uniformly bounded on any $G$-metric ball in $\Imm^l(M,\bR^d)$.
By \eqref{eqn:GammaTensor}, the analogous estimates hold for $\Gamma$ of all types.
\end{lemma}
\begin{proof}
The assumption $\|h\|_{G_f} \gtrsim \|\nabla h\|_{L^{\infty}(g)}$ implies by \cref{lem:Linftyassumption} and \cref{rem:volumedens} that the $L^p(g)$ and $L^p(\ol{g})$ norms are uniformly equivalent on metric balls, hence in the following we will use the $L^p(g)$ norm in order to expedite the calculations.

Let $f \in H^1((0,1),\Imm^l(M,\bR^d))$ be again a path of immersions with velocity $\dot f$.
Since $S=\nabla Tf$ and $\ol{\nabla}$ is independent of $t$, a direct calculation shows that
\[
\pl_t S=\pl_t\nabla Tf=\pl_t(\ol\nabla+\ChristCovec)T f=\nabla^2\pl_tf+\pl_t\ChristCovec(T f),
\]
and thus, using \eqref{eqn:normMetric}, we have
\[
\begin{split}
|\pl_t S|_{g} \le |\nabla^2 \dot f|_g + \sqrt{2}|\pl_t\ChristCovec|_{g}
\end{split}
\]
Taking the $p$th power and integrating, we obtain
\begin{equation}\label{eqn:pl_tS_estimate_1}
\|\pl_t S\|_{L^p(\ol{g})} \lesssim  \| \nabla^2 \dot f \|_{L^p(g)} + \|\pl_t \ChristCovec||_{L^p(g)}.
\end{equation}

We now evaluate $\|\pl_t \ChristCovec||_{L^p(g)}$.
Since $\ol\nabla$ is independent of the immersion $f$ (and consequently of time $t$) we have that $\partial_t \ChristCovec =\partial_t\nabla$ (note that while $\nabla$ is a differential operator of order $1$, its time derivative is tensorial, so the equality makes sense).
The following (implicit) formula for the   variation of the covariant derivative has been derived in~\cite[Lemma 5.8]{bauer2011sobolev},
\begin{equation}\label{eq:dt_nabla}
g((\partial_t\nabla)(X,Y),Z)=\tfrac12(\nabla\pl_t g)\left(X\otimes Y\otimes Z+Y\otimes X\otimes Z-Z\otimes X\otimes Y\right).
\end{equation}
Here $X,Y$ and $Z$ are vector fields on $M$, and we write $\nabla(X,Y)=\nabla_XY$.
Using the variation formula for the metric~\eqref{eq:variationmetric} we obtain
\begin{align}\label{eq:nablaptg}
 \nabla \pl_t  g&= \nabla 2\sym \langle\nabla \dot f,Tf\rangle_{\bR^d} =2\sym \langle\nabla^2 \dot f,Tf\rangle_{\bR^d} +
2\sym \langle\nabla \dot f,\nabla Tf\rangle_{\bR^d}\\
&=2\sym \langle\nabla^2 \dot f,Tf\rangle_{\bR^d} +
2\sym \inner{\nabla \dot f}{S}.
\end{align}
We thus have
\begin{align}
 \|\nabla \pl_t  g\|_{L^p(g)}&\leq 2\left\|\sym \inner{\nabla^2 \dot f}{Tf}\right\|_{L^p(g)} +
2\left\|\sym \inner{\nabla \dot f}{S}\right\|_{L^p(g)} \\
    &\overset{\eqref{eqn:normMetric}}{\le}2^{3/2}\|\nabla^2 \dot f\|_{L^p(g)} + 2\|\nabla \dot f\|_{L^{\infty}(g)}\|S\|_{L^p(g)}.
\end{align}
Using \eqref{eq:dt_nabla} and the fact that $g^{-1}$ is uniformly bounded on metric balls, we obtain
\begin{align}
\label{eqn:variationGammaBound}
\|\pl_t\ChristCovec \|_{L^p(g)}\lesssim \|\nabla \pl_t  g\|_{L^p(g)}&\lesssim\|\nabla^2\dot f\|_{L^p( g)}+\|\nabla\dot f\|_{L^\infty(g)} \|S\|_{L^p(g)}.
 \end{align}
Combining with \eqref{eqn:pl_tS_estimate_1} we therefore obtain, using our assumptions on $G_f$, that
\[
\|\pl_tS\|_{L^p(g)} \lesssim \|\nabla^2\dot f\|_{L^p( g)}+\|\nabla\dot f\|_{L^\infty(g)} \|S\|_{L^p(g)} \lesssim (1+ \|S\|_{L^p(g)})\|\dot f\|_{G_f},
\]
and thus, using again \cref{rem:volumedens},
\[
\|\pl_tS\|_{L^p(\ol{g})} \lesssim (1+ \|S\|_{L^p(\ol{g})})\|\dot f\|_{G_f},
\]
which implies the result for $S$ using \cref{lem:Gronwall}.
The result for $\ChristCovec$ follows similarly from \eqref{eqn:variationGammaBound}, using the boundedness of $\|S\|_{L^p(g)}$ on metric balls.
The result for $\ChristVec$ is analogous.

Finally, the Lipschitzness of $\H$ follows from the Lipschitzness of $S$ and of $g^{-1}$ (which follows by \cref{lem:Linftyassumption}), as $\H = \Tr(g^{-1}S)$.
\end{proof}
In the case of $H^3$-immersions we need one more estimate which will allow us to control the highest order derivatives appearing in the metric.
\begin{lemma}[Christoffel Symbol and Second Fundamental Form Derivative Bounds]
\label{lem:H3plusW24plusLinftyassumption}
Let $l\geq 3$ and let $G$ be a smooth Riemannian metric on
$\Imm^l(M,\bR^d)$
such that
\[
\|h\|_{G_f}
 \gtrsim
\| \nabla^3 h \|_{L^2(g)} + \| \nabla^2 h \|_{L^4(g)}+ \|\nabla h\|_{L^{\infty}(g)}.
\]
Then the functions
\begin{align}
\ol\nabla S &:\left( \Imm^l(M,\bR^d),G\right)\to \left(H^{l-3}(M,T^*M^{\otimes 3}\otimes\bR^d), \|\cdot\|_{L^2(\ol{g})}\right),\\
\ol\nabla\Christ{*} &:\left( \Imm^l(M,\bR^d),G\right)\to \left(H^{l-3}(M, T^*M^{\otimes 3}\otimes TM),  \|\cdot\|_{L^2(\ol{g})}\right),\qquad *\in\{\Vec,\Covec\},
\end{align}
are Lipschitz continuous and thus uniformly bounded on any $G$-metric ball in $\Imm^l(M,\bR^d)$.
By \eqref{eqn:GammaTensor}, the analogous estimates hold for $\Gamma$ of all types.
\end{lemma}

\begin{proof}
Let $f \in H^1((0,1),\Imm^l(M,\bR^d))$ be a path of immersions with velocity $\dot f$. 
Following the proof of \cref{lem:W2pplusLinftyassumption} we have
\begin{align}
\pl_t \ol\nabla S
&=\ol\nabla \pl_t S
=\ol\nabla \nabla^2 \dot f +\ol{\nabla}(\pl_t \Gamma Tf)
=\nabla^3 \dot f-\Gamma\nabla^2\dot f + (\pl_t \ol{\nabla}\Gamma)Tf + \pl_t \Gamma\ol\nabla Tf\\
&=\nabla^3 \dot f-\Gamma\nabla^2\dot f + (\pl_t \ol{\nabla}\Gamma)Tf + \pl_t \Gamma S - \pl_t \Gamma\, \Gamma\, Tf,
\end{align}
where, for the sake of simplicity, we omitted the type of $\Gamma$ as well as how it contracts with the tensor it acts on.
Hence we estimate, using \eqref{eqn:normMetric} and H\"older's inequality,
\begin{align}
\|\pl_t \ol\nabla S\|_{L^2(\ol g)}
    & \leq \|\nabla^3 \dot f\|_{L^2(g)}+\|\Gamma\|_{L^4(g)}\|\nabla^2\dot f\|_{L^4(g)}  +   \|\pl_t \Gamma\|_{L^4(g)} \|S\|_{L^4(g)}\\
    &\quad+\sqrt{2} \|\pl_t \Gamma\|_{L^4(g)}\|\Gamma\|_{L^4(g)} + \sqrt{2}\|\pl_t \ol{\nabla}\ChristCovec\|_{L^2(g)}.
    \end{align}
Using our assumptions on $G_f$,  \cref{lem:W2pplusLinftyassumption} (which implies uniform boundedness of $\|\Gamma\|_{L^4(\ol g)}$ and $\|S\|_{L^4(\ol g)}$ on metric balls), estimate \eqref{eqn:variationGammaBound} and the equivalence of the $\g$ and $g$ norms, we obtain that, uniformly on any metric ball,
\begin{align}\label{eq:pl_t_nabla_S_estimate_1}
    \|\pl_t \ol\nabla S\|_{L^2(g)} \lesssim \|\dot f\|_{G_f} + \|\pl_t \ol{\nabla}\ChristCovec\|_{L^2(g)}.
\end{align}

We now estimate $\|\pl_t \ol{\nabla}\ChristCovec\|_{L^2(g)}$.
Using again that $\ol\nabla$ is independent of the immersion $f$ (and consequently of time $t$) we have
\begin{equation}\label{eq:pl_t_nabla_Gamma_1}
   \partial_t \ol\nabla\ChristCovec =\ol\nabla\pl_t \ChristCovec =\ol\nabla\partial_t\nabla=\nabla\partial_t\nabla-\Gamma\partial_t\nabla.
\end{equation}
Using \eqref{eqn:variationGammaBound} and \cref{lem:W2pplusLinftyassumption} (which implies uniform boundedness of $\|\Gamma\|_{L^4(\ol g)}$ on the metric ball), the last term can be bounded via
\begin{equation}\label{eq:pl_t_nabla_Gamma_2}
\|\Gamma\partial_t\nabla\|_{L^2( g)}
\leq\|\Gamma\|_{L^4( g)}\|\partial_t\nabla\|_{L^4( g)}
=\|\Gamma\|_{L^4( g)}\|\partial_t\ChristCovec \|_{L^4( g)}
\lesssim\|\dot f\|_{G_f}.
\end{equation}
For the other term, $\nabla\partial_t\nabla$, we apply $\nabla$ on both sides of \eqref{eq:dt_nabla} to obtain
\begin{align}
&g((\nabla\partial_t\nabla)(X,Y),Z)
+R_1\\
&=\nabla[g((\partial_t\nabla)(X,Y),Z)]
=\nabla[\tfrac12(\nabla\pl_t g)\left(X\otimes Y\otimes Z+Y\otimes X\otimes Z-Z\otimes X\otimes Y\right)]\\
&=\tfrac12(\nabla\nabla\pl_t g)\left(X\otimes Y\otimes Z+Y\otimes X\otimes Z-Z\otimes X\otimes Y\right)+R_2,
\end{align}
where $R_1,R_2$ are summands depending linearly and nontrivially on $(\nabla X,\nabla Y,\nabla Z)$.
As $\nabla\partial_t\nabla$ is tensorial, we must have $R_1=R_2$
(which one can also verify by explicit calculation using again \eqref{eq:dt_nabla}).
Since $g^{-1}$ is uniformly Lipschitz in $L^\infty(\g)$ on metric balls (by \cref{lem:Linftyassumption} and the assumption that $\|h\|_{G_f} \gtrsim  \|\nabla h\|_{L^{\infty}(g)}$),
we obtain
 \begin{equation}\label{eq:nablaGamma_terms}
     \|\nabla\partial_t\nabla\|_{L^2(g)}\sim\|\nabla\nabla\partial_t g\|_{L^2(g)}.
 \end{equation}
Using equation~\eqref{eq:nablaptg} we calculate
 \begin{align}
\nabla\nabla \pl_t  g&=\nabla\left(2\sym \inner{\nabla^2 \dot f}{Tf} +
2\sym \inner{\nabla \dot f}{S}\right)\\
&=2\sym \inner{\nabla^3 \dot f}{Tf} +4\sym \inner{\nabla^2 \dot f}{S}+2\sym \inner{ \nabla \dot f}{\nabla S}.
\end{align}
We estimate the $L^2$-norm of each of the three terms separately.
For the first term we calculate, using \eqref{eqn:normMetric},
\begin{align}
&\|2\sym \inner{\nabla^3 \dot f}{Tf}\|_{L^2( g)}\lesssim
\|\nabla^3 \dot f\|_{L^2(g)}.
\end{align}
The second term can be estimated as
\begin{align}
&\|4\sym \inner{\nabla^2 \dot f}{S}\|_{L^2( g)}\lesssim
\|\nabla^2 \dot f\|_{L^4( g)}\|S\|_{L^4( g)}\lesssim
\|\nabla^2 \dot f\|_{L^4(g)}
\end{align}
using the uniform boundedness of $\|S\|_{L^4(\ol g)}$ on metric balls from \cref{lem:W2pplusLinftyassumption}.
For the third term we calculate
\begin{align}
&\|2\sym \inner{ \nabla \dot f}{\nabla S}\|_{L^2( g)}\leq
2\|\nabla \dot f\|_{L^{\infty}(g)}\|\nabla S \|_{L^2( g)}\\
&\qquad
= 2\|\nabla \dot f\|_{L^{\infty}(g)}\| \ol \nabla S +\Gamma S\|_{L^2( g)}\\
&\qquad\leq
2\|\nabla \dot f\|_{L^{\infty}( g)}\left(\| \ol \nabla S\|_{L^2( g)} +\|\Gamma S\|_{L^2( g)}\right)\\&\qquad
\leq 2\|\nabla \dot f\|_{L^{\infty}( g)}\left(\| \ol \nabla S\|_{L^2( g)} +\|\Gamma\|_{L^4( g)} \|S\|_{L^4( g)}\right) \\
    &\qquad\lesssim (1+ \| \ol \nabla S\|_{L^2( g)})\|\nabla \dot f\|_{L^{\infty}( g)}.
\end{align}
Combining these three estimates with \eqref{eq:pl_t_nabla_Gamma_1}, \eqref{eq:pl_t_nabla_Gamma_2} and \eqref{eq:nablaGamma_terms}, we obtain
\begin{equation}\label{eq:pl_t_nabla_Gamma_3}
    \|\partial_t\ol\nabla\ChristCovec\|_{L^2(g)} \lesssim (1+ \| \ol \nabla S\|_{L^2( g)})\|\dot f\|_{G_f} ,
\end{equation}
which, together with \eqref{eq:pl_t_nabla_S_estimate_1},
implies 
\begin{equation}
    \|\partial_t\ol\nabla S\|_{L^2(g)} \lesssim (1+ \| \ol \nabla S\|_{L^2( g)})\|\dot f\|_{G_f} ,
\end{equation}
hence the result for $\ol\nabla S$ follows from \cref{lem:Gronwall}.
Once we know that $\ol\nabla S$ is uniformly bounded on metric balls, \eqref{eq:pl_t_nabla_Gamma_3} and \cref{lem:Gronwall} imply the result for $\ChristCovec$, from which the results for other types of $\Gamma$ follow.
\end{proof}
The above estimates are sufficient to prove completeness for the mean-curvature-weighted $H^3$-metrics. To obtain the result for higher order metrics, we need additional estimates on higher order derivatives of the metric and of $\Gamma$, which we present in the next two lemmas.

\begin{lemma}[Metric Controls Lipschitz Constants of Derivatives]
\label{lem:HkplusW2inftyplusLinftyassumption2}
Let $l\geq k>3$ and let $G$ be a smooth Riemannian metric on 
$\Imm^l(M,\bR^d)$
such that
\[
\|h\|_{G_f}
 \gtrsim
\| \nabla^k h \|_{L^2(g)} + \| \nabla^2 h \|_{L^\infty(g)}+ \|\nabla h\|_{L^{\infty}(g)}.
\]
Then we have, uniformly on every metric ball,
that
\begin{align}
\|h\|_{G_f}
 &\gtrsim
\| \nabla^i h \|_{L^2(g)}
&&\text{for }1\leq i\leq k,\\
\|h\|_{G_f}
 &\gtrsim
\| \nabla^i h \|_{L^\infty(g)}
&&\text{for }1\leq i\leq k-2.
\end{align}
\end{lemma}
\begin{proof}
By \cref{lem:Linftyassumption} we have that the control on $ \|\nabla h\|_{L^{\infty}(g)}$ implies, in any metric ball, control on $ \|\nabla h\|_{L^{2}(g)}$.
Thus, by \cref{lem:metricdom} we have, uniformly on a metric ball, $\|h\|_{G_f}\gtrsim\|\nabla^jh\|_{L^2(g)}$ for $1\leq j\leq k$.
Furthermore, by \cref{sqrtvol,lem:W2pplusLinftyassumption} we have that $\Vol$ and $|\H|$ are uniformly controlled on metric balls.
The claim now follows from \cref{cor:Sobembedding} for $q=2$, applied to $\nabla^i h$.
\end{proof}

\begin{lemma}[Christoffel Symbol and Second Fundamental Form Higher Derivative Bounds]
\label{lem:HkplusW2inftyplusLinftyassumption}
Let $l\geq k>3$ and let $G$ be a smooth Riemannian metric on 
$\Imm^l(M,\bR^d)$
such that
\begin{align}
\|h\|_{G_f}
 \gtrsim
\| \nabla^k h \|_{L^2(g)} + \| \nabla^2 h \|_{L^{\infty}(g)}+ \|\nabla h\|_{L^{\infty}(g)}.
\end{align}
Then, for $i\in\{k-3,k-2\}$ and $*\in\{\Vec,\Covec\}$ the functions
\begin{align}
\nabla^i S &:\left( \Imm^l(M,\bR^d),G\right)\to \left(H^{l-i-2}(M,T^*M^{\otimes(i+2)}\otimes\bR^d), \|\cdot\|_{L^2(\ol{g})}\right),\\
\nabla^i \Christ{*} &:\left( \Imm^l(M,\bR^d),G\right)\to \left(H^{l-i-2}(M,T^*M^{\otimes(i+2)}\otimes TM), \|\cdot\|_{L^2(\ol{g})}\right)
\end{align}
and, for $i\leq k-4$ and $*\in\{\Vec,\Covec\}$ the functions
\begin{align}
\nabla^i S &:\left( \Imm^l(M,\bR^d),G\right)\to \left(H^{l-i-2}(M,T^*M^{\otimes(i+2)}\otimes\bR^d), \|\cdot\|_{L^\infty(\ol{g})}\right),\\
\nabla^i\Christ{*} &:\left( \Imm^l(M,\bR^d),G\right)\to \left(H^{l-i-2}(M,T^*M^{\otimes(i+2)}\otimes TM), \|\cdot\|_{L^\infty(\ol{g})}\right)
\end{align}
are Lipschitz continuous and thus uniformly bounded (from above) on any $G$-metric ball in $\Imm^l(M,\bR^d)$.
By \eqref{eqn:GammaTensor}, the analogous estimates hold for $\Gamma$ of all types.
\end{lemma}

\begin{proof}
We start by noting that for a tensor $X$ we have
\begin{equation}\label{eq:pl_t_nabla_i}
\begin{split}
\partial_t \nabla^i X 
    &= \nabla^i \partial_t X + \sum_{j=1}^i \sum_{m=0}^{j-1}P_{ijm}(\Gamma,\nabla)\partial_t(\nabla^m \Gamma) \nabla^{i-j}X \\
    &= \nabla^i \partial_t X + \sum_{j=1}^i \sum_{m=0}^{j-1} Q_{ijm}(\Gamma,\nabla)(\nabla^{m}\partial_t \Gamma) \nabla^{i-j}X,
\end{split}
\end{equation}
where $P_{ijm}$ and $Q_{ijm}$ are homogeneous polynomials of degree $j-m-1$ in $\Gamma$ and $\nabla$ that produce a tensor (e.g., for degree $2$ it is a linear combination of $\Gamma\Gamma$ and $(\nabla\Gamma)$, but not of $\Gamma\nabla$ or $\nabla^2$ or $\nabla(\Gamma\cdot)$).
This can be easily obtained by induction (writing $\partial_t\nabla^{i+1}X=\partial_t\nabla(\nabla^iX)$), using that $\partial_t \nabla Y= \nabla\pl_t Y + (\partial_t\Gamma) Y $ for any tensor $Y$ and an appropriate type of $\Gamma$.
Keep in mind that above we simply write $\Gamma$ even though each occurrence of $\Gamma$ (e.g., in the summands of $P_{ijm}(\Gamma,\nabla)$) may be of different type.
Also note that in \eqref{eq:pl_t_nabla_i} we did not indicate explicitly which indices contract in the products of tensors, since this is irrelevant for the remaining argument;
however, be aware that even the summands of each single polynomial $P_{ijm}(\Gamma,\nabla)$ differ in how they act on $\partial_t(\nabla^m \Gamma)$ and $\nabla^{i-j}X$ (analogously for $Q_{ijm}(\Gamma,\nabla)$).

Second, applying $\nabla^i$ to both sides of \eqref{eq:dt_nabla} and arguing as in the proof of \cref{lem:H3plusW24plusLinftyassumption}, we obtain that, uniformly on metric balls, we have
\begin{equation}\label{eq:nabla_i_pl_t_Gamma}
\|\nabla^i \partial_t \Gamma\|_{L^p} \sim \|\nabla^{i+1}\partial_t g\|_{L^p},
\end{equation}
using that $g^{-1}$ is uniformly bounded on metric balls (by \cref{lem:Linftyassumption} and the assumption that $\|h\|_{G_f} \gtrsim  \|\nabla h\|_{L^{\infty}(g)}$).

Third, applying $\nabla^i$ to both sides of \eqref{eq:variationmetric}, we obtain
\begin{equation}
    \label{eq:nabla_i_pl_t_g}
    \nabla^i \pl_t g = 2\sym\inner{\nabla^{i+1}\pl_t f}{Tf} + 2\sum_{j=1}^i \binom{i}{j}\sym \inner{\nabla^{i-j+1}\pl_t f}{\nabla^{j-1}S}.
\end{equation}

Fourth, applying $\nabla^i$ to $\pl_t S = \nabla^2\dot f+\pl_t \Gamma \,Tf$, we obtain that
\begin{equation}
    \nabla^i\pl_t S = \nabla^{i+2}\dot f + \sum_{j=0}^{i-1}\binom{i}{j}\nabla^j \pl_t\Gamma\, \nabla^{i-1-j}S + \nabla^i \pl_t \Gamma\, Tf.
\end{equation}
Hence, using \eqref{eqn:normMetric},
\begin{align}
    \label{eq:nabla_i_pl_t_S}
    \|\nabla^i\pl_t S\|_{L^p} & \lesssim \|\nabla^{i+2}\dot f\|_{L^p} + \sum_{j=0}^{i-2} \|\nabla^j \pl_t\Gamma\|_{L^\infty} \|\nabla^{i-1-j}S\|_{L^p}\\
        &\quad + \|\nabla^{i-1} \pl_t\Gamma\|_{L^p}\|S\|_{L^\infty} + \|\nabla^i \pl_t \Gamma\|_{L^p}.
\end{align}

The proof now follows by induction in $i$, using that, for $m\leq k-2$, $\|\nabla^m\partial_t f\|_{L^\infty(g)}\lesssim\|\partial_t f\|_{G_f}$ uniformly on a metric ball by \cref{lem:HkplusW2inftyplusLinftyassumption2}:
The base case $i=0$ is provided by \cref{lem:W2pplusLinftyassumption}.
Assume the bounds are established for $l<i$ for some $i$.
Assume first that $i\le k-4$, then \eqref{eq:nabla_i_pl_t_Gamma} and \eqref{eq:nabla_i_pl_t_g} imply that on each metric ball we have
\begin{align}\label{eq:nabla_i_dot_g_estimate}
\|\nabla^{i+1}\pl_t g\|_{L^\infty} 
    &\lesssim \|\nabla^{i+2} \dot f\|_{L^\infty} + \sum_{j=1}^{i+1} \|\nabla^{i+2-j}\dot f\|_{L^\infty} \|\nabla^{j-1} S\|_{L^\infty} \lesssim  (1+ \|\nabla^iS\|_{L^\infty})\|\dot f\|_{G_f},
\end{align}
where we used the assumptions on $G_f$ as well as the induction hypothesis for $j\le i$.
Similarly, using that $\|\nabla^{k-1}\dot f\|_{L^2(g)},\|\nabla^k\dot f\|_{L^2(g)}\lesssim\|\dot f\|_{G_f}$ uniformly on a metric ball by \cref{lem:HkplusW2inftyplusLinftyassumption2}, for $i\in \{k-3,k-2\}$ we obtain the same estimate with $L^2$ instead of $L^\infty$ on the leftmost and rightmost side.
Combining this estimate with \eqref{eq:nabla_i_pl_t_Gamma}, \eqref{eq:nabla_i_pl_t_S}, \cref{lem:HkplusW2inftyplusLinftyassumption2}, and the induction hypothesis yields
\begin{align}
    \|\nabla^i\pl_t S\|_{L^p} \lesssim (1+ \|\nabla^iS\|_{L^p})\|\dot f\|_{G_f},
\end{align}
where $p=\infty$ for $i\le k-4$ and $p=2$ for $i\in \{k-3,k-2\}$.
We now apply \eqref{eq:pl_t_nabla_i} to $X=S$, obtaining, for the same values of $p$,
\begin{align}
    \|\pl_t\nabla^i S\|_{L^p} \lesssim (1+ \|\nabla^iS\|_{L^p})\|\dot f\|_{G_f},
\end{align}
which proves the induction step for $S$ using \cref{lem:Gronwall}.
The induction step for $\Gamma$ is then obtained similarly, using \eqref{eq:pl_t_nabla_i} for $X=\Gamma$ and the induction step for $S$.
\end{proof}

\begin{remark}[Improved Bounds Below Top Derivatives]
    The assumptions of \cref{lem:HkplusW2inftyplusLinftyassumption2} actually also imply that on any metric ball $\|h\|_{G_f} \gtrsim \| \nabla^{k-1} h \|_{L^p(g)}$ for any $p\in[1,\infty)$.
    Thus, the assumptions of \cref{lem:HkplusW2inftyplusLinftyassumption} imply that for $i=k-3$ the bounds are in $L^p$, $p\in [1,\infty)$, and not only for $p=2$.
    However, we will not need these improved bounds in the following.
\end{remark}

\section{Completeness for abstract Riemannian metrics on the Space of Surfaces}\label{sec:completeness}
In this Section we present the main results of the article --- completeness properties on the space of surfaces --- for an abstract Riemannian metric $G$ under the assumption that $G$ bounds certain Sobolev norms. We will then construct in~\cref{sec:curvature_weighted}  Sobolev  metrics of order $k\geq 3$, that satisfy this assumption.
\subsection{Geodesic and Metric Completeness}
We start by proving geodesic and metric completeness. To this end we will need the following assumption for the metric $G$, which we will use throughout the remainder of this section:
\begin{assumption}\label{assumption_completeness}
Let $k\geq 3$ and let $G$ be a smooth Riemannian metric on $\Imm^k(M,\bR^d)$ such 
that on any metric ball we have
\begin{align}\label{eq:metric_completeness_cond}
\|h\|_{G_f}
 \gtrsim
\| \nabla^k h \|_{L^2(g)}+\| \nabla^2 h \|_{L^p(g)}+ \|\nabla h\|_{L^{\infty}(g)}+\|h \|_{L^2(g)}
\end{align}
with $p=4$ for $k=3$ and $p=\infty$ else.
\end{assumption}
The completeness result for the abstract Riemannian metric is then given as follows:
\begin{theorem}[Metric and Geodesic Completeness for General Riemannian Metrics]\label{thm:metric_completeness_abstract}
Let $k\geq 3$ and let $G$ be a smooth Riemannian metric on $\Imm^k(M,\bR^d)$ that satisfies \cref{assumption_completeness}. We have:
\begin{enumerate}[label=(\alph*)]
\item  
$(\Imm^k(M,\bR^d),G)$ is metrically and geodescially complete.
\item 
If the metric $G$ is in addition invariant under the action of $\Diff(M)$, then the geodesic completeness continues to hold on $\Imm^l(M,\bR^d)$ for  any $k\leq l\in\bN$ and thus, in particular, on the space of smooth immersions $\Imm(M,\bR^d)$.
\end{enumerate}
\end{theorem}
\begin{proof}
To prove the first statemtent we aim to show that the Riemannian manifold $(\mathcal M,G)=(\Imm^k(M,\bR^d),G)$ satisfies the conditions \ref{assumption1}--\ref{assumption2} of \cref{thm:abstract_completeness} with $(\ol{\mathcal M},\ol G)=(H^k(M,\bR^d),\ol G^k)$ for $\ol G^k$ from \eqref{equation:backgroundmetric}.
The main difficulty lies in showing that the metric $G$ dominates the background $H^k$-metric on metric balls.

In \cref{rem:volumedens} we already discussed that
on every $G$-metric ball in $\Imm^k(M,\bR^d)$ there exists $C>0$ such that for every $f$ in the ball the volume density is bounded by $C$, i.e.,
\begin{equation}
\|\vol/\ol{\vol}\|_{L^\infty}=\|\rho\|_{L^\infty}>C, 
\end{equation}
which shows that condition~\ref{assumption2} of \cref{thm:abstract_completeness} is satisfied. It remains to prove condition~\ref{assumption1}, i.e., to show that  the background $H^k$-metric $\ol G^k$ from~\eqref{equation:backgroundmetric} is uniformly upper bounded by $G$ on metric balls.

We first treat the case $k=3$ and will deal with general $k>3$ afterwards. 
To bound the $L^2$-term of $\ol G$ by the $L^2$-term of $G$ one only needs to control the volume density of $f$, which follows directly from \cref{rem:volumedens}. It remains to look at the highest order term, i.e., we consider the $H^3$-term only. Using \cref{lem:Linftyassumption} we can bound the volume density $\rho$ and (finite-dimensional) Riemannian metric $g$ on metric balls leading to
\begin{align}
&\int_{M}\ol g(\ol \nabla^3 h,\ol\nabla^3 h)\ol \vol 
 \lesssim \int_{M} g(\ol \nabla^3 h,\ol\nabla^3 h) \vol.
\end{align}
Next we convert the covariant derivative $\ol \nabla$ of the background metric $\ol{g}$ to the covariant derivative $\nabla$ of the metric $g$ using the tensors $\Gamma$ as introduced in~\eqref{eq:Gamma},
\begin{align}
&\int_{M} g(\ol \nabla^3 h,\ol\nabla^3 h) \vol=
\int_{M} g(\ol \nabla^2 Th,\ol\nabla^2 Th) \vol\\&\quad
=\int_{M} g\left(\ol \nabla ((\nabla-\Gamma) Th),\ol\nabla((\nabla-\Gamma) Th)\right) \vol\\&
\quad\lesssim 
\int_{M} g\left(\ol \nabla (\nabla Th),\ol\nabla (\nabla Th)\right) \vol
+\int_{M} g\left(\ol \nabla (\Gamma Th),\ol\nabla(\Gamma Th)\right) \vol.\label{equation:backgroundmetricbounds0}
\end{align}
 We treat the two terms separately. To bound the first term we repeat the same procedure  to obtain
\begin{align} 
&\int_{M} g\left(\ol \nabla (\nabla Th),\ol\nabla (\nabla Th)\right) \vol
\\&\quad\lesssim \int_{M} g\left(\nabla^2 Th,\nabla^2 Th\right) \vol
+\int_{M} g\left(\Gamma\nabla Th,\Gamma \nabla Th\right) \vol\\&\quad
\lesssim G_f(h,h)+\int_{M} g\left(\Gamma\nabla Th,\Gamma \nabla Th\right) \vol.\label{equation:backgroundmetricbounds}
\end{align}
By \cref{lem:W2pplusLinftyassumption}, our assumption on the metric $G$ implies that $\Gamma$ is uniformly bounded in $L^4$ on metric balls, and thus, using
Cauchy--Schwarz on the second term  of~\eqref{equation:backgroundmetricbounds} yields
\begin{align} 
&\int_{M} g\left(\Gamma\nabla Th,\Gamma \nabla Th\right) \vol\label{equation:backgroundmetricbounds1}\\&\leq
\left(\int_{M} g\left(\nabla Th, \nabla Th\right)^2 \vol\right)^{1/2}
\left(\int_{M} g(\Gamma, \Gamma)^2 \vol\right)^{1/2}\\&
\lesssim 
\left(\int_{M} g\left(\nabla Th, \nabla Th\right)^2 \vol\right)^{1/2}\lesssim G_f(h,h).
\end{align}
To bound the second term in~\eqref{equation:backgroundmetricbounds0} we calculate 
\begin{align}
&\int_{M} g\left(\ol \nabla (\Gamma Th),\ol\nabla (\Gamma Th)\right) \vol\\
&\quad= \int_{M} g\left((\ol \nabla \Gamma) Th+\Gamma\ol \nabla Th,(\ol \nabla \Gamma) Th+\Gamma\ol \nabla Th\right) \vol\\
&\quad\lesssim \int_{M} g\left((\ol \nabla \Gamma) Th,(\ol \nabla \Gamma) Th\right) \vol
+\int_{M} g\left(\Gamma\ol \nabla Th,\Gamma\ol \nabla Th\right) \vol\\
&\quad\lesssim \int_{M} g\left((\ol \nabla \Gamma) Th,(\ol \nabla \Gamma) Th\right) \vol+\int_{M} g\left(\Gamma\nabla Th,\Gamma\nabla Th\right) \vol\\&\qquad\qquad\qquad
+\int_{M} g\left(\Gamma\Gamma Th,\Gamma\Gamma Th\right) \vol.
\end{align}
The second term is similar to~\eqref{equation:backgroundmetricbounds1} (only the type of $\Gamma$ and its contraction with $\nabla Th$ differ) and can be bounded in the same way. Note that here $\Gamma$ acts only on part of the $(0,2)$ tensor $\nabla Th$ and that the notation $\Gamma \nabla Th$ is thus ambiguous, but this ambiguity is irrelevant for our estimates.
For the third term we use again \cref{lem:W2pplusLinftyassumption} to control $\Gamma$ in $L^4$ and our assumptions on $G$ to control $\|Th\|_{L^\infty}$ by $\|h\|_{G_f}$.
It remains to control the first term. Using again the assumption that we control $Th$ in $L^{\infty}$, the problem reduces to the control of $\|\ol \nabla \Gamma\|_{L^2(g)}\sim\|\ol \nabla \Gamma\|_{L^2(\ol g)}$, which holds by \cref{lem:H3plusW24plusLinftyassumption} and our assumptions on the metric $G$. This concludes the proof of condition~\ref{assumption1} for $k=3$.

Next we treat the case $k>3$. Note that the proof also holds for $k=3$ and that the previous presentation of this special case was merely for readability.
Similar to the case $k=3$, we only need to
bound the highest order term as the $L^2$-term can be handled using \cref{lem:Linftyassumption}. We start by investigating $\ol \nabla^kh=(\nabla-\Gamma)^{k-1}T h$. We claim that, for any $k$, the expression $(\nabla-\Gamma)^{k-1}T h$ is a sum of terms of the form
\begin{equation}\label{eq:termsnablak}
\underbrace{\nabla^{i_1}(\Gamma)\ldots\nabla^{i_m}(\Gamma)}_{m\text{ terms}}\nabla^{k-1-i-m}Th=\underbrace{\nabla^{i_1}(\Gamma)\ldots\nabla^{i_m}(\Gamma)}_{m\text{ terms}}\nabla^{k-i-m}h,
\end{equation}
where
\begin{equation} 
0\leq m\leq k-1,\quad  0\leq i\leq k-1-m, \quad\text{and } i_1+\ldots+i_m=i.
\end{equation} 
We want to emphasize that $\Gamma$ in each of these terms can be of different type, which for the sake of readability we do not indicate,
and that we do not indicate either how each occurring tensor acts on the remaining terms (i.e.\ along which dimensions all tensors contract).
However, as before, this ambiguity is irrelevant for the remaining argument.

Using induction in $k$, one can easily see that this claim is true: In the base case $k=1$ one simply obtains $Th$ so that the claim holds trivially. For the induction step we  need to apply $\nabla-\Gamma$ to a term of the form~\eqref{eq:termsnablak} and check that this yields again terms of the correct type. This follows by a straightforward calculation.

Using  again \cref{lem:Linftyassumption} we can then estimate
\begin{align}
&\int_{M}\ol g(\ol \nabla^k h,\ol\nabla^k h)\ol \vol 
 \lesssim \int_{M} g((\nabla-\Gamma)^{k-1} Th,(\nabla-\Gamma)^{k-1}T h) \vol.
\end{align}
Applying Young's inequality we can bound the $L^2(g)$-norm of $(\nabla-\Gamma)^{k-1} Th$ by the sum of the $L^2(g)$-norms of terms of the form~\eqref{eq:termsnablak}. To bound these terms we proceed by investigating them term by term depending on the value of $m$: For $m=0$ (and thus $i=0$) we get exactly the reparametrization-invariant $H^k$-term $\|\nabla^k h\|_{L^2(g)}$, which can be bounded by assumption. For $m=1$ and $i=0$  we can estimate via
\[
\|\Gamma\nabla^{k-1}h\|^2_{L^2(g)}\leq \|\Gamma\|^2_{L^{\infty}(g)}\|\nabla^{k-1}h\|^2_{L^2(g)}\lesssim \|\Gamma\|^2_{L^{\infty}(g)}\|h\|^2_{G_f},
\]
where we can control the $L^{\infty}$-norm of $\Gamma$ via \cref{lem:W2pplusLinftyassumption} and the $L^2$-norm of $\nabla^{k-1}h$ via \cref{lem:HkplusW2inftyplusLinftyassumption2}.  For the remaining terms we have that $i+m\geq 2$
and thus we can control the $L^{\infty}$-norm of $\nabla^{k-i-m}h$ using \cref{lem:HkplusW2inftyplusLinftyassumption2}.
Thus to control these terms we only need to control  \begin{equation}\|\nabla^{i_1}(\Gamma)\ldots\nabla^{i_m}(\Gamma)\|^2_{L^2(g)}.
\end{equation}
If $m=1$, then the above term reduces to $\|\nabla^{i_1}(\Gamma)\|^2_{L^2(g)}$,
which we can control by~\cref{lem:HkplusW2inftyplusLinftyassumption} since $i_1=i\leq k-2$ (for $i\leq k-4$ note that any $L^2$-norm can be controlled by the $L^\infty$ norm via H\"older's inequality and \cref{sqrtvol}). If $m=2$ then $i_1\leq k-4$ and $i_2\leq k-3$ (or vice versa) since
$i_1+i_2\leq k-3$. Without loss of generality we assume that $i_1\leq i_2$ and estimate
\begin{equation}
\|\nabla^{i_1}(\Gamma)\nabla^{i_2}(\Gamma)\|^2_{L^2(g)}\leq
\|\nabla^{i_1}(\Gamma)\|^2_{L^{\infty}(g)}\|\nabla^{i_2}(\Gamma)\|^2_{L^2(g)},
\end{equation} 
which is bounded by~\cref{lem:HkplusW2inftyplusLinftyassumption}.
Finally for $m\geq 3$ we estimate 
\begin{equation}\|\nabla^{i_1}(\Gamma)\ldots\nabla^{i_m}(\Gamma)\|^2_{L^2(g)}\leq
\|\nabla^{i_1}(\Gamma)\|^2_{L^{\infty}(g)}\ldots\|\nabla^{i_m}(\Gamma)\|^2_{L^{\infty}(g)},
\end{equation}
which is bounded again by~\cref{lem:HkplusW2inftyplusLinftyassumption}, since $i_j\leq k-4$ for all $1\leq j\leq m$.
Thus we have concluded the proof of condition~\ref{assumption1} for $k>3$. From here the proof of metric and geodesic completeness on $\Imm^k(M,\bR^d)$ follows by invoking \cref{thm:abstract_completeness}.

The remaining statements on geodesic completeness on $\Imm^l(M,\bR^d)$ for $l\geq k$ follow directly from the $\Diff(M)$-invariance of the metric $G$ by applying an Ebin--Marsden type no-loss-no-gain result~\cite{ebin1970groups}, cf.~\cite{bruveris2017regularity,bauer2020fractional} for versions that are applicable in the context of the present article.
\end{proof}

\subsection{Existence of Minimizing Geodesics}\label{sec:existence}
Next we discuss the existence of minimizing geodesics for  Riemannian metrics on $\Imm^k(M,\bR^d)$. By requiring an additional property for the Riemannian metric this will turn out as an immediate consequence of the abstract result \cref{thm:abstract_completeness}, using the estimates we have already proved for the completeness result. We first formulate this additional assumption for the abstract Riemannian metric $G$:
\begin{assumption}\label{assumption_convexity}
Let $k\geq 3$ and let $G$ be a smooth Riemannian metric on $\Imm^k(M,\bR^d)$ such
that the induced path energy $E$
is sequentially lower semicontinuous with respect  to weak convergence in $H^1((0,1),H^k(M,\bR^d))$.
\end{assumption}
Using this existence of minimizing geodesics can be proven as follows:
\begin{theorem}[Existence of Minimizing Geodesics for General Riemannian Metrics]\label{thm:existence_abstract}
Let $k\geq 3$ and let $G$ be a smooth Riemannian metric on $\Imm^k(M,\bR^d)$ that satisfies~\cref{assumption_completeness,assumption_convexity}.
Then there exists a minimizing geodesic in $(\Imm^k(M,\bR^d),G)$ between
any two surfaces from the same connected component of $\Imm^k(M,\bR^d)$.
\end{theorem}

\begin{proof}
    By \cref{thm:metric_completeness_abstract}, 
    the metric $G$ satisfies conditions \ref{assumption1}--\ref{assumption2} from \cref{thm:abstract_completeness}.
    By assumption, condition \ref{assumption3} is also satisfied.
    Thus, the result immediately follows from \cref{thm:abstract_completeness}.
\end{proof}

\subsection{Induced Results on the Shape Space of Unparametrized Surfaces}
In this section we will study completeness properties of the space of unparametrized surfaces $\Shape^k(M,\bR^d)$ equipped with the quotient distance of the geodesic distance of a reparametrization-invariant Riemannian metric on $\Imm^k(M,\bR^d)$.
As explained in~\cref{sec:diffgroup_action} this quotient space does not carry the structure of a manifold, and thus we do not obtain a Riemannian metric on it, but we will only view it as a metric space.  The first main result of this section is presented in the following theorem:
\begin{theorem}[Metric Completeness of the Shape Space $\Shape^k(M,\bR^d)$]\label{thm:metric_complete_shape}
Let $k\geq 3$ and let $G$ be a smooth Riemannian metric on $\Imm^k(M,\bR^d)$ that satisfies~\cref{assumption_completeness}. Assume in addition that $G$ is invariant under the action of $\Diff^k(M)$.
Then $(\Shape^k(M,\bR^d),\dist)$ is a complete metric space, where $\dist$
denotes the quotient distance of the geodesic distance $\dist_G$ on $\Imm^k(M,\bR^d)$.
\end{theorem}
To prove this result we first recall the following classic Lemma from metric geometry, cf.~\cite{burago2001course,bruveris2015completeness}:
\begin{lemma}[Completeness of Quotient Spaces]\label{lem:abstract_quotientspace}
Let $(\mathcal M, \dist)$ be a metric space upon which a group $\Group$ acts by isometries.
If the quotient space $\mathcal M / \Group$  is Hausdorff, then the quotient distance
\begin{equation}
\dist(\Group \cdot x, \Group \cdot y) :
= \inf_{h \in \Group} \dist(x,h\cdot y)
\end{equation}
defines a metric on $\mathcal M / \Group$ that is compatible with the quotient topology on $\mathcal M / \Group$. If $(\mathcal M, \dist)$ is complete, then so is $(\mathcal M / \Group, \dist)$.
\end{lemma}
Using this result the proof of metric completeness of $\Shape^k(M,\bR^d)$ is immediate:
\begin{proof}[Proof of \cref{thm:metric_complete_shape}]
First we recall that the space $\Shape^k(M,\bR^d)$ is Hausdorff by~\cite{inci2013diffeo}. Furthermore, the action of $\Diff^k(M)$ on $\Imm^k(M,\bR^d)$ is by isometries, since we assumed that the Riemannian metric (and thus the corresponding geodesic distance) is invariant under this action. Finally, by \cref{thm:metric_completeness_abstract} the space $(\Imm^k(M,\bR^d),G)$ is metrically complete; recall that metric completeness for a Riemannian manifold means that the manifold is metrically complete viewed as a metric space with the induced geodesic distance function. Thus the desired result follows by \cref{lem:abstract_quotientspace}.
\end{proof}
Next we prove the existence of minimizing geodesics and optimal reparametrizations. Here we will denote an equivalence class of surfaces via $[f]:=f\circ\Diff^k(M)$ for $f\in \Imm^k(M,\bR^d)$.  
\begin{theorem}[Minimizing Geodesics on the Shape Space $\Shape^k(M,\bR^d)$]\label{thm:metric_convex_shape}
Let $k\geq 3$ and let $G$ be a smooth Riemannian metric on $\Imm^k(M,\bR^d)$ that satisfies~\cref{assumption_completeness,assumption_convexity}. Assume in addition that $G$ is invariant under the action of $\Diff^k(M)$. Then we have:
\begin{enumerate}[label=(\alph*)]
    \item {\bf Existence of Optimal Reparametrizations:} For any $[f_1], [f_2] \in \Shape^k(M, \mathbb{R}^d)$ in the same connected component
there exists an optimal reparametrization $\bar \varphi\in \Diff^k(M)$
attaining the infimum in the definition of the quotient distance,
\begin{equation}
\dist([f_1], [f_2])
  = \inf_{\varphi \in \Diff^k(M)} 
    \dist_G(f_1, f_2 \circ \varphi)=
    \dist_G(f_1, f_2 \circ \bar \varphi).
\end{equation}
\item {\bf Shape Space is a Geodesic Length Space:} For any $[f_1], [f_2] \in \Shape^k(M, \mathbb{R}^d)$ in the same connected component there exists a connecting minimizing geodesic in $\Shape^k(M, \mathbb{R}^d)$.
\end{enumerate}
\end{theorem}

\begin{proof}
    First, we note that (b) follows immediately from (a) and the existence of minimizing geodesics in $\Imm^k(M,\bR^d)$ as proved in \cref{thm:existence_abstract}.
    To prove (a), fix $[f_1]$ and $[f_2]$ in the same connected component, and consider $\varphi_n\in\Diff^k(M)$ such that
    \[\dist([f_1], [f_2])
  = \lim_{n\to \infty} 
    \dist_G(f_1, f_2 \circ \varphi_n).\]
    Consider paths $f^n\in H^1((0,1), \Imm^k(M,\bR^d))$ such that $f^n(0) = f_1$ and $f^n(1) = f_2 \circ \varphi_n$ and such that
    \[
    \lim_{n\to \infty} E(f^n) = \dist([f_1], [f_2])^2.
    \]
 Thus, it is clear that $f^n(t)$, for all $n$ and $t$, lie in a large enough metric ball centered at $f_1$.
    Arguing as in \cref{thm:abstract_completeness}, we obtain that $f^n$ is bounded in $H^1((0,1),H^k(M,\bR^d))$.
    Thus $f^n$ converges weakly in this space to some $f\in H^1((0,1),H^k(M,\bR^d))$, hence also in $C^{0,\frac12}([0,1],H^k(M,\bR^d))$, and, by the Arzel\`a--Ascoli Theorem and the compact embedding of $H^k(M,\bR^d)$ into $H^{k-\varepsilon}(M,\bR^d)$ for all $\varepsilon>0$, up to extracting a subsequence we even have $f^n\to f$ strongly in $C^0([0,1],H^{k-\varepsilon}(M,\bR^d))$.
    In particular we have $f(0) = f_1$ and $f(1) = \lim_n f^n(1)=\lim_nf_2\circ\varphi_n$, where the limit is taken in $H^{k-\varepsilon}$.
    As a consequence, $[f(1)]$ is arbitrarily close to $[f_2\circ\varphi_n]=[f_2]$ in $\Shape^{k-\epsilon}(M,\bR^d)$.
    Since the space $\Shape^{k-\varepsilon}(M,\bR^d)$ is Hausdorff \cite[Proposition 6.2]{bruveris2015completeness}, we obtain that
    $f(1) = f_2 \circ \varphi$ for some $\varphi \in \Diff^{k-\epsilon}(M)$.
    The proof is done if we show that $\varphi \in \Diff^k(M)$.
    Note that $f(1), f_2 \in H^k$ (since $f(1)$ is the $H^k$-weak limit of $f^n(1)$).
    Thus, it is sufficient to show that (locally) $f_2^{-1} : f_2(M) \to M$ is an $H^k$-map, since $\varphi = f_2^{-1}\circ f(1)$ and the composition of $H^k$ maps stays in $H^k$.
    Here, we need to work locally as $f_2$ might not be an embedding.
    
    This claim follows by the same standard argument in differential geometry as is employed for the fact that, for a $C^k$-embedding $f:M\to \bR^d$, $f(M)$ is a $C^k$-manifold (with respect to the ambient atlas defined by graphs over the tangent spaces in $\bR^d$) and the inverse of $f$ is also $C^k$.
    The only additional ingredient needed is that the inverse of an invertible $H^k(\bR^d,\bR^d)$ map, for $k>\frac{d}{2} + 1$, is also in $H^k$, as shown in \cite{inci2013diffeo}.
\end{proof}

\section{Completeness properties of the Curvature-Weighted $H^k$-metric}\label{sec:curvature_weighted}
In this Section, we will construct a family of reparametrization-invariant mean-curvature-weighted $H^k$-metrics for $k\geq 3$ that satisfy both \cref{assumption_completeness,assumption_convexity}, which in turn directly leads to a proof of our main result~\cref{theorem:main}.
\begin{definition}[Mean-Curvature-Weighted $H^k$-Metric]\label{def:metric}
For  $l\geq k\geq 3$ and $f\in \Imm^l(M,\bR^d)$ and $h_1,h_2\in T_f\Imm^l(M,\bR^d)$ we consider the mean-curvature-weighted $H^k$ inner product   given by
\begin{equation}\label{eq:curvatureweightedmetric}
\begin{split}
G^k_f(h_1,h_2)
= \int_{M} h_1\cdot h_2+|{\H}|^4 &\left(g(\nabla h_1,\nabla h_2) +g(\nabla^2 h_1 ,\nabla^2 h_2)\right)+ g(\nabla^k h_1,\nabla^k h_2)\,\vol ,
\end{split}
\end{equation}
where $g, \H, \vol$ and $\nabla$ all depend on the surface $f$.
\end{definition}
First, we note that $G^k_f$ is well-defined for $l\ge k\ge 3$, that is, the terms in the integrand are integrable.
Indeed, if $l\ge 4$ then by the Sobolev embedding $H^{l-2}(M)\subset L^\infty(M)$ it follows that the mean curvature $\H$ is uniformly bounded, from which the integrability of the middle terms in $G^k_f$ immediately follows with \cref{thm:derivativeRegularity}.
In the case $l=k=3$, we have that $H^{l-2} = H^1\subset L^p$ for any $p<\infty$, hence both the mean curvature $\H$ and the terms $\nabla^2 h_i$ are in, say, $L^6$ so that the most difficult term $|\H|^4g(\nabla^2 h_1,\nabla^2 h_2)$ is integrable (note that we always have that $g$ and $\rho=\vol/\ol\vol$ are in $L^\infty$).

Next we show that the Riemannian metric $G^k$ satisfies the two main assumptions, which we needed to prove completeness and existence of minimizing geodesics:
\begin{theorem}\label{lemma:assumptions}
Let $M$ be a two-dimensional, closed manifold, let $d,k\geq 3$ and let $\Imm^k(M,\bR^d)$ be the space of immersed surfaces equipped with the curvature-weighted $H^k$-metric $G^k$ from \eqref{eq:curvatureweightedmetric}. We have:
\begin{enumerate}[label=(\alph*)]
\item\label{item:smoothness_metric}  $G^k$ is a \emph{smooth} Riemannian metric, which is invariant under the action of the diffeomorphism group $\Diff^k(M)$ as defined in~\cref{sec:diffgroup_action}.
\item \label{item:assumptions_metric}$G^k$ satisfies both \cref{assumption_completeness} and \cref{assumption_convexity}.
\end{enumerate}
\end{theorem}

\begin{proof}
{\bf \cref{item:smoothness_metric}:}
The invariance of $G^k_f$ under the action of the diffeomorphism group follows directly by an application of the transformation formula for integrals
as well as the fact $(\nabla_f h)\circ\varphi=\nabla_{f\circ\varphi}(h\circ\varphi)$ for any $\varphi\in\Diff(M)$ (where the index to $\nabla$ indicates what immersion the covariant derivative is associated with).
Thus it only remains to show that $G^k$ depends smoothly on the foot-point $f$, which is non-trivial as all the terms in the definition of $G^k$ --- e.g., the covariant derivative $\nabla$ and the mean curvature $\H$ --- depend highly nonlinearly on the immersion $f$. 
For Sobolev type metrics, even of fractional order, this has been studied in detail in previous work~\cite{muller2017applying,bauer2020fractional}, and the metric~\eqref{eq:curvatureweightedmetric} satisfies the conditions derived there. 
This is, however, somewhat difficult to see, and for the reader's convenience we give a direct proof in the following. To this end we first rewrite the inner product as
\[
\begin{split}
G^k_f(h_1,h_2) 
= \int_{M} \bigg(&g(h_1,h_2)+|\H|^4g(\overline\nabla h_1,\overline\nabla h_2)\\&+|\H|^4g(\nabla \overline\nabla h_1,\nabla\overline\nabla h_2)+ g(\nabla^{k-1}\overline\nabla h_1,\nabla^{k-1}\overline\nabla h_2)\bigg)\rho \overline{\vol} ,
\end{split}
\]
where we used that $\nabla$ acting on functions does not depend on the immersion $f$ and thus we have $\nabla h=\overline{\nabla}h$ for any $h\in T_f\Imm^k(M,\bR^d)$. Next we study the individual components of this inner product:  The mapping $f\mapsto \rho$ is a smooth mapping from $\Imm^k(M,\bR^d)$ to
$H^{k-1}(M,\bR^d)$ by~\cite[Lemma~3.1]{bauer2020fractional}; similarly the mapping $f\mapsto\H$ is smooth as a mapping from $\Imm^k(M,\bR^d)$ to
$H^{k-2}(M,\bR^d)$.
Finally, by~\cite[Lemma 3.5]{bauer2022smooth} the mapping
$f\mapsto \nabla$ is smooth as a mapping from $\Imm^k(M,\bR^d)$ to $L(H^s(M,E),H^{s-1}(M,T^*M\otimes E))$ for any vector bundle $E$ over $M$ and $s\in [0,k-1]$. Note that we cannot choose $s=k$, but this is not a problem since we rewrote the inner product before to exchange the first $\nabla$ with a $\overline{\nabla}$. From here the result follows by the module properties of Sobolev spaces, cf.~\cite{behzadan2021multiplication}.

\noindent{\bf \cref{item:assumptions_metric}, \cref{assumption_completeness}:}
By definition the metric $G^k$ includes the first and last term on the right-hand side of condition~\eqref{eq:metric_completeness_cond} and thus trivially bounds these two terms.
The third term, $\|\nabla h\|_{L^{\infty}(g)}$, follows by applying the Sobolev embedding theorem, as stated in \cref{cor:Sobembedding}, to $\nabla h$,
\begin{align}
\| \nabla h\|_{L^{\infty}} \leq C \sqrt{\operatorname{Vol}} \left( \|\nabla^3 h\|_{L^2(g)} + \||\H|^2 \nabla h\|_{L^2(g)}
 \right)\lesssim G^k_f(h,h),
\end{align}
where we used in the last step that the total volume is controlled on metric balls by~\cref{sqrtvol}.

For the remaining term, $\|\nabla^2 h\|_{L^p(g)}$, we distinguish again between $k=3$ and $k>3$. For $k=3$ \cref{thm:L4toH1} implies
\begin{align}
\|\nabla^2 h\|_{L^4(g)}&\leq C\left(\|\nabla^2 h\|_{L^2(g)}+\|\sqrt{|\H|}\nabla^2 h\|_{L^2(g)}+\|\nabla^3 h\|_{L^2(g)} \right)\\
&\leq C\left(\|\nabla^2 h\|_{L^2(g)}+\|\sqrt{1+|\H|^4}\nabla^2 h\|_{L^2(g)}+\|\nabla^3 h\|_{L^2(g)} \right)\\&\lesssim G^3_f(h,h),
\end{align}
where we used that $\sqrt{(1+|\H|^4)}\geq \sqrt{|\H|}$ as well as \cref{lem:metricdom}.
For $k>3$ we apply the same argument as we used above for $\| \nabla h\|_{L^{\infty}}$ to bound $\| \nabla^2 h\|_{L^{\infty}}$.

\noindent{\bf \cref{item:assumptions_metric}, \cref{assumption_convexity}:}
To show the weak sequential lower semicontinuity of $f\mapsto E(f)$, let us abbreviate, for any $l,\alpha,p$,
\begin{gather}
L^p=L^p(M,N),\
H^l=H^l(M,N),\
W^{l,p}=W^{l,p}(M,N),\\
L^2L^p=L^2((0,1),L^p),\
L^2H^l=L^2((0,1),H^l),\
H^1H^l=H^1((0,1),H^l),\\
C^0L^p=C^0([0,1],L^p),\
C^0H^l=C^0([0,1],H^l),\
C^0W^{l,p}=C^0([0,1],W^{l,p}),\\
C^0C^{0}=C^0([0,1],C^{0}(M,N)),\
C^{0,\alpha}H^l=C^{0,\alpha}([0,1],H^l),
\end{gather}
where the codomain $N$ should always be clear from the context (most often $N=\bR^d$)
and the norms of function spaces on $M$ are always with respect to the background metric $\ol g$ unless otherwise specified.

Next we consider a sequence $f^n\rightharpoonup f$ weakly in $H^1H^k$.
By $\|h_t-h_s\|_{H^k}=\|\int_s^t\dot h_r\,\mathrm d r\|_{H^k}\leq\int_s^t\|\dot h_r\|_{H^k}\mathrm d r\leq\sqrt{|t-s|}\|\dot h\|_{L^2H^k}$ (using H\"older's inequality in the last step),
we have the continuous embedding $H^1H^k\hookrightarrow C^{0,1/2}H^k$.
Hence also $f^n\rightharpoonup f$ weakly in $C^{0,1/2}H^k$.
Furthermore, since by Sobolev embedding $H^k$ embeds compactly into $W^{k-1,p}$ for any $p\geq1$,
we may apply the Arzel\`a--Ascoli Theorem to obtain another subsequence (still indexed by $n$) with $f^n\to f$ strongly in $C^0W^{k-1,p}$.

Let us denote the (time-dependent) metric, volume density, covariant derivative, and mean curvature on surface $f^n$ by $g^n,\rho^n,\nabla_n,\H_n$ (analogously for $f$).

Since $Tf^n\to T f$ in $C^0W^{k-2,p}$ and since $W^{k-2,p}$ forms a Banach algebra for $p$ large enough \cite[V Theorem 5.23]{adams2003sobolev},
we have $g^n=\langle Tf^n\cdot,Tf^n\cdot\rangle_{\bR^d}\to\langle Tf\cdot,Tf\cdot\rangle_{\bR^d}=g$ in $C^0W^{k-2,p}$.
As in the proof of \cref{lem:approximation_MSSestimates} \ref{enm:metricConvergence} this implies convergence of the extended metric (on arbitrary tensor bundles)
\begin{equation}
g^n\to g
\qquad\text{in }C^0W^{k-2,p}\hookrightarrow C^0C^0
\end{equation}
(and the same for its inverse, exploiting that inverting an operator is smooth for operators with eigenvalues bounded away from zero).
Consequently, also
\begin{equation}
\rho^n\to\rho
\qquad\text{in }C^0C^0.
\end{equation}

We also require convergence of the change of metric operator that turns $\ol g$ into $g^n$.
In detail, note that the operator $\ol g^{-1}g^n:TM\to TM$ is positive definite $\ol g$-selfadjoint,
thus has a unique positive definite $\ol g$-selfadjoint square root $R^n:TM\to TM$ that depends smoothly on $g^n$.
It is straightforward to check $\ol g(R^nX,R^nY)=g^n(X,Y)$ for any tangent vectors $X,Y\in TM$.
Analogously one can define $R^n$ as an endomorphism on $(TM)^{\otimes i}\otimes(T^*M)^{\otimes j}\otimes(\bR^d)^{\otimes m}$ with same properties.
The above convergence of $g^n$ implies
\begin{equation}
R^n\to R
\qquad\text{in }C^0W^{k-2,p}\hookrightarrow C^0C^0.
\end{equation}

Next we consider the convergence of $\Gamma_n=\nabla_n-\ol\nabla$.
Working in a chart, $\Gamma_n$ can be expressed as the difference of the Christoffel symbols
(of the second kind,
see \cref{sec:inducedGeometry}) associated with $g^n$ and $\ol g$,
where with a slight abuse of notation we also use $g^n,\ol g$ for the respective coordinate representations of the metrics.
The Christoffel symbols of the second kind are sums of products of $(g^n)^{-1}$, which converge to $g^{-1}$ in $C^0W^{k-2,p}$, with derivatives of $g^n$, wich converge in $C^0W^{k-3,p}$.
If $k>3$, then $W^{k-3,p}$ is a Banach algebra for $p$ large enough \cite[V Theorem 5.23]{adams2003sobolev} so that these Christoffel symbols of $g^n$ converge to those of $g$ in $C^0W^{k-3,p}$.
If $k=3$, then $(g^n)^{-1}\to g^{-1}$ in $C^0W^{1,p}\hookrightarrow C^0L^\infty$, while the derivatives of $g^n$ converge in $C^0L^p$.
In either case we arrive at
\begin{equation}
\Gamma_n,\Gamma\in C^0H^{k-2}
\qquad\text{with}\qquad
\Gamma_n\to\Gamma
\quad\text{in }C^0W^{k-3,p}
\end{equation}
if $\Gamma$ acts on tangent vectors and analogously if it acts on arbitrary tensors.

Now consider the mean curvature.
We have $\nabla_nTf^n=\ol\nabla Tf^n+\Gamma_nTf^n$, of which the first summand converges to $\ol\nabla Tf$ in $C^0W^{k-3,p}\hookrightarrow C^0L^p$
and the second to $\Gamma Tf$ in $C^0L^{p/2}$ (as the product of two functions converging in $C^0L^p$).
Consequently, we have convergence of the second fundamental forms $S_n=\nabla_nTf^n\to\nabla Tf=S$ in $C^0L^{p/2}$.
Together with the uniform convergence $(g^n)^{-1}\to g^{-1}$ this implies
\begin{equation}
\H_n=\Tr((g^n)^{-1}S_n)\to\Tr(g^{-1}S)=\H
\qquad\text{in }C^0L^{p/2}.
\end{equation}

Finally, we have
\begin{equation}
\nabla_n\dot f^n
=\ol\nabla\dot f^n
\rightharpoonup\ol\nabla\dot f
=\nabla\dot f
\qquad\text{weakly in }L^2H^{k-1}\hookrightarrow L^2L^p,
\end{equation}
as well as (using $\ol\nabla^2\dot f^n\rightharpoonup\ol\nabla^2\dot f$ weakly in $L^2H^{k-2}\hookrightarrow L^2L^p$ and $\ol\nabla\dot f^n\rightharpoonup\ol\nabla\dot f$ weakly in $L^2H^{k-1}\hookrightarrow L^2L^p$,
while $\Gamma_n\to\Gamma$ strongly in $C^0L^p$)
\begin{equation}
\nabla_n^2\dot f^n=\ol\nabla^2\dot f^n+\Gamma_n\ol\nabla \dot f^n
\rightharpoonup\ol\nabla^2\dot f+\Gamma\ol\nabla\dot f=\nabla^2\dot f
\qquad\text{weakly in }L^2L^{p/2}.
\end{equation}
Likewise,
\begin{equation}
\nabla_n^k\dot f^n
\rightharpoonup\nabla^k\dot f
\qquad\text{weakly in }L^2L^2,
\end{equation}
which can be obtained as follows:
Expanding $\nabla_n^k\dot f^n=(\ol\nabla+\Gamma_n)^k\dot f^n=(\ol\nabla+\Gamma_n)^{k-1}T\dot f^n$ and applying the product rule for derivatives of products and contractions of tensors,
one obtains a large sum of products between a (higher order) derivative of $\dot f^n$ and (higher or zero order) derivatives of multiple $\Gamma_n$.
Exactly one summand, $\ol\nabla^k\dot f^n$, contains no $\Gamma_n$, and it converges weakly to $\ol\nabla^k\dot f$ in $L^2L^2$.
Again exactly one summand, $[\ol\nabla^{k-2}\Gamma_n]T\dot f^n$, contains only a single derivative of $\dot f^n$, and it converges weakly to $[\ol\nabla^{k-2}\Gamma]T\dot f$ in $L^2L^2$ as can be seen as follows:
Due to $\ol\nabla^{k-2}\Gamma_n\in C^0L^2$ and $T\dot f^n\in L^2C^0$ one can readily check that the product lies in $L^2L^2$.
Now for any smooth $\phi$ we have
\begin{align}
&\int_0^1\int_M\phi(\ol\nabla^{k-2}\Gamma_n)T\dot f^n\ol\vol\,\mathrm dt
=\int_0^1\int_M\ol\nabla^*(\phi T\dot f^n)\ \ol\nabla^{k-3}\Gamma_n\ol\vol\,\mathrm dt\\
&\to\int_0^1\int_M\ol\nabla^*(\phi T\dot f)\ \ol\nabla^{k-3}\Gamma\ol\vol\,\mathrm dt
=\int_0^1\int_M\phi(\ol\nabla^{k-2}\Gamma)T\dot f\ol\vol\,\mathrm dt,
\end{align}
since $T\dot f^n\rightharpoonup T\dot f$ weakly in $L^2H^{k-1}$, $\phi$ is smooth, and $\ol\nabla^{k-3}\Gamma_n\to\ol\nabla^{k-3}\Gamma$ strongly in $C^0L^p$.
Finally, all other summands contain in total at most $k-3$ derivatives of (multiple) $\Gamma_n$ (which all converge in $C^0L^p$) and at most $k-1$ derivatives of $\dot f^n$ (which converge weakly in $L^2L^p$),
so all other summands converge weakly in $L^2L^2$ for $p$ large enough.

In summary, we obtain
\begin{align}
\liminf_{n\to\infty}E(f^n)
&=\liminf_{n\to\infty}\|\sqrt{\rho^n}\dot f^n\|_{L^2L^2}^2
+\|\sqrt{\rho^n}|\H_n|^2R^n\nabla_n\dot f^n\|_{L^2L^2}^2\\
&\quad+\|\sqrt{\rho^n}|\H_n|^2R^n\nabla_n^2\dot f^n\|_{L^2L^2}^2
+\|\sqrt{\rho^n}R^n\nabla_n^k\dot f^n\|_{L^2L^2}^2\\
&\geq\|\sqrt\rho \dot f\|_{L^2L^2}^2
+\|\sqrt\rho|\H|^2R\nabla \dot f\|_{L^2L^2}^2\\
&\quad+\|\sqrt\rho|\H|^2R\nabla^2\dot f\|_{L^2L^2}^2
+\|\sqrt\rho R\nabla^k\dot f\|_{L^2L^2}^2
=E(f),
\end{align}
since the argument of each norm converges weakly in $L^2L^2$ and the norm is weakly lower semicontinuous.
\end{proof}
\begin{remark}[Other Choices for $G^k$]\label{rem:otherchoices}
From the above proof we note that the metric $G^k$ is not the optimal choice in terms of the chosen powers of the mean curvature weight, e.g.\ for the $H^3$-case we could choose
\begin{align}
\tilde G^3_f(h_1,h_2) 
= \int_{M} g(h_1,h_2)&+|\H|^4g(\nabla h_1,\nabla h_2)+f(\H)g(\nabla^2 h_1,\nabla^2 h_2)+ g(\nabla^3 h_1,\nabla^3 h_2)\vol,
\end{align}
where $f(x):\bR^{d}\to[0,\infty)$ is any smooth function that grows at least like $|x|$,
and would still satisfy both \cref{assumption_completeness} and \cref{assumption_convexity}.
\end{remark}

\begin{remark}[Existence Conditions]
The coercivity condition from \cref{assumption_completeness} will ensure sequential compactness of the set of finite energy paths between $f_0$ and $f_1$.
This and the lower semicontinuity condition, \cref{assumption_convexity}, represent the typical ingredients of a variational existence proof.
Examples where the lower semicontinuity condition is violated and thus existence of minimizing geodesics is not expected to hold include metrics $G_f$ that depend in a nonconvex way on $\nabla^kf$,
such as $$G_f(h,h)=G^k_f(h,h)\left(1+\int_M1-\exp[-(1-|\nabla^kf|^2)^2]\,\vol\right),$$ which encourages $|\nabla^kf|=1$
(for concrete counterexamples to existence of minimizing geodesics it may actually be easier to replace, in the last integral, $f$ by its first component and $\nabla$ by $\nabla_X$ for a fixed given vector field $X$ on $M$).
\end{remark}

Using the above our main result immediately follows:
\begin{proof}[Proof of~\cref{theorem:main}]
We have shown in~\cref{lemma:assumptions} that $G^k$ is a smooth,  reparametrization-invariant metric on $\Imm^k(M,\bR^d)$ that satisfies \cref{assumption_completeness,assumption_convexity}. Thus all statements of the theorem follow by invoking~\cref{thm:metric_completeness_abstract,thm:existence_abstract,thm:metric_complete_shape,thm:metric_convex_shape}.
\end{proof}

\appendix
\section{Table of Notation}\label{app:table_notation}
Below we summarize some of the main notations used in the paper.
These notations are explained in \cref{sec:spaces}.
  \begin{longtable}{@{}p{.3\textwidth}p{\dimexpr.7\textwidth-2\tabcolsep\relax}@{}}
  \hline
  \mbox{General Notation} \\
  \hline
  $M$ & a two-dimensional, closed manifold; the parameter space\\
  $\Imm(M,\bR^d)$ & the space of smooth immersions in $\bR^d$ with $d\geq 3$ \\
  $\Imm^l(M,\bR^d)$ & the space of $H^l$ immersions in $\bR^d$ with $l>2$ \\
  $\Shape^l(M,\bR^d)$& the space of unparametrized $H^l$ immersions in $\bR^d$\\
    $f\in \Imm^l(M,\bR^d)$ & an immersion\\
    $[f]\in \Shape^l(M,\bR^d)$& an unparametrized immersed surface\\
    $h\in T_f \Imm^l(M,\bR^d)$ & a tangent vector to the space of immersions 
    \\
    $\langle\cdot,\cdot\rangle_{\bR^d}$ & the Euclidean inner product on $\bR^d$\\
    $\Tr$& trace of an endomorphism
    \\\hline
  \mbox{Surface Geometry} \\\hline
  $\ol{g}$& a fixed background Riemannian metric on $M$ \\
$\overline{\vol}$& the volume density of $\ol{g}$ \\
 $\overline{\nabla}$& the Levi-Civita derivative of $\ol{g}$ \\
 $g=f^*\langle\cdot,\cdot\rangle$ & the pullback metric of $f$ \\
 $\vol$& the volume density of $f$  \\
 $\nabla$& the Levi-Civita derivative of $f$  \\ 
 $\Vol$& total volume of $f$  \\
   $S=\nabla Tf$ & the vector-valued second fundamental form of $f$ \\
   $\H=\Tr(g^{-1}S)$ & the vector-valued mean curvature of $f$ \\
   $\Gamma=\nabla-\overline{\nabla}$& deviation from the background covariant derivative\\
  $\ChristVec,\ChristCovec$&particular instances of $\Gamma$ acting on tangent and cotangent vectors, respectively
\\\hline
  \mbox{Riemannian Metrics on $\Imm^l(M,\bR^d)$} \\\hline
  $G^k$ & mean-curvature-weighted, reparametrization-invariant $H^k$-metric as defined in \cref{def:metric}\\
$\ol G^k$ & background $H^k$-metric as defined in \cref{def:backgroundmetric}\\
$\dist_G$ & distance metric induced by Riemannian metric $G$\\
$\dist$ & distance in a metric space
\\\hline
  \mbox{Induced Norms on $T_f\Imm^l(M,\bR^d)$} \\\hline
 $\|h\|_{G_f}$ & norm of $h\in T_f\Imm^l(M,\bR^d)$ w.r.t.\ a metric $G$\\ 
   $\|h\|_{L^2(g)}$ & $L^2$-norm of $h\in T_f\Imm^l(M,\bR^d)$ w.r.t.\ $\vol$\\
   $\|h\|_{L^2(\ol{g})}$ & $L^2$-norm of $h\in T_f\Imm^l(M,\bR^d)$ w.r.t.\ $\overline{\vol}$\\
   \hline
   \end{longtable}

\bibliographystyle{abbrv}  

\end{document}